\documentclass[12pt]{article}

\usepackage[T1]{fontenc}
\usepackage[margin = 1in]{geometry}

\usepackage{amsthm, amssymb, amstext}
\usepackage{xspace}
\usepackage[fleqn]{amsmath}
\usepackage{latexsym}
\usepackage[dvips]{graphicx}
\usepackage{comment}
\usepackage{hyperref}
\usepackage[capitalize]{cleveref}
\usepackage{mathtools}
\usepackage{enumerate} 
\usepackage{paralist}
\usepackage{enumitem}
\usepackage{bm}

\usepackage{todonotes}
\usepackage{tikz}
\usetikzlibrary{calc}
\usetikzlibrary{fit}
\usetikzlibrary{graphs, graphs.standard, positioning}
\usetikzlibrary{decorations.pathmorphing}
\usetikzlibrary{backgrounds}

\crefname{claim}{Claim}{Claims}
\crefname{conjecture}{Conjecture}{Conjectures}
\crefname{question}{Question}{Questions}
\crefname{figure}{Figure}{Figures}

\newcommand{\lca}{\textsf{lca}}

\usepackage{setspace}
\makeatletter
\newtheorem*{rep@theorem}{\rep@title}
\newcommand{\newreptheorem}[2]{%
\newenvironment{rep#1}[1]{%
 \def\rep@title{#2 \ref{##1}}%
 \begin{rep@theorem}}%
 {\end{rep@theorem}}}
\makeatother

\newreptheorem{theorem}{Theorem}
\newreptheorem{lemma}{Lemma}
\newreptheorem{corollary}{Corollary}

\newtheorem{theorem}{Theorem}
\newtheorem{lemma}[theorem]{Lemma}

\newtheorem{corollary}[theorem]{Corollary}

\newtheorem{claim}{Claim}

\theoremstyle{definition}

\newcommand\tw{\text{tw}}

\renewcommand\leq{\leqslant}

\renewcommand\geq{\geqslant}

\newcommand\nega{negative\xspace}
\newcommand\posi{positive\xspace}

\newcommand\cro{\mathfrak{c}\xspace}

\definecolor{myRed}{rgb}{0.68, 0.05, 0.0}
\colorlet{myBlue}{blue!70!black}
\colorlet{myViolet}{myBlue!55!myRed}
\definecolor{darkraspberry}{rgb}{0.53, 0.15, 0.34}
\definecolor{olive}{rgb}{0.42, 0.56, 0.14}
\hypersetup{
	colorlinks=true,
        linkcolor=myBlue, 
        citecolor=myBlue,
	bookmarksopen=true,
	bookmarksnumbered,
	bookmarksopenlevel=2,
	bookmarksdepth=3
}
\usepackage{breakurl}
\usepackage{authblk}

\newcommand{\footnoteref}[1]{%
  \hyperref[#1]{\textsuperscript{\ref*{#1}}}%
}

\title{Excluding a Forest Induced Minor}

\author[1]{\'Edouard Bonnet\footnote{É.\ Bonnet has been supported by the French National Research Agency through the project TWIN-WIDTH with reference number ANR-21-CE48-0014.}}
\author[1]{Benjamin Duhamel}
\author[2]{Robert Hickingbotham\footnote{R.\ Hickingbotham is supported by the Belgian National Fund for Scientific Research (FNRS).}}

\affil[1]{CNRS, ENS de Lyon, Université Claude Bernard Lyon 1, LIP, UMR 5668, Lyon, France}

\affil[2]{D\'epartement d'Informatique, Universit\'e libre de Bruxelles, Belgium}

\date{}

\begin{document}

\maketitle

\begin{abstract}
  In the first paper of the \emph{Graph Minors} series [JCTB '83], Robertson and Seymour proved the Forest Minor theorem: the $H$-minor-free graphs have bounded pathwidth if and only if $H$ is a~forest.
  In recent years, considerable effort has been devoted to understanding the unavoidable induced substructures of graphs with large pathwidth or large treewidth.
  In this paper, we give an induced counterpart of the Forest Minor theorem: for any $t \geqslant 2$, the $K_{t,t}$-subgraph-free\footnote{Most induced analogues of graph minor results tend to be more meaningful and interesting within weakly sparse classes, i.e., $K_{t,t}$-subgraph-free for some fixed~$t$. It is straightforward to derive from known facts that the class of $H$-induced-minor-free graphs has bounded pathwidth (resp.~treewidth) if and only if $H$ is a~clique on at~most~two (resp.~at~most~four) vertices.\label{fn:ws}} $H$-induced-minor-free graphs have bounded pathwidth if and only if $H$ belongs to a~class $\mathcal F$ of forests, which we describe as the induced minors of two (very similar) infinite parameterized families.  

  This constitutes a significant step toward classifying the graphs $H$ for which every weakly sparse $H$-induced-minor-free class has bounded treewidth.
  Our work builds on the theory of constellations developed in the \emph{Induced Subgraphs and Tree Decompositions} series. 
\end{abstract}

\section{Introduction}

In this paper, we fully classify the unavoidable induced minors in weakly sparse\footnoteref{fn:ws} graph classes with unbounded pathwidth.
Pathwidth and treewidth are central parameters in algorithmic and structural graph theory.
A substantial body of work has been devoted to understanding the unavoidable substructures in graphs with large pathwidth/treewidth.
For the subgraph and minor relations, the situation is well understood.
The Excluded Forest Minor Theorem~\cite{RS1} states that a~graph class $\mathcal G$ has bounded pathwidth if and only if $\mathcal G$ excludes some forest as a~minor.
Likewise, the Excluded Grid Minor Theorem~\cite{ROBERTSON198692} states that a~graph class $\mathcal G$ has bounded treewidth if and only if $\mathcal G$ excludes some planar graph as a~minor. 

In contrast, the picture for induced substructures is more subtle.
First, any graph class~$\mathcal G$ of bounded pathwidth or treewidth must exclude both $K_t$ and $K_{t,t}$ as induced subgraphs for some $t$.
By Ramsey's theorem, this is equivalent to $\mathcal G$ being \emph{weakly sparse}, meaning that it excludes $K_{t',t'}$ as a~subgraph for some $t'$.
Second, for $\mathcal G$ to have bounded pathwidth, it must exclude some forest as an induced minor.
However, these two conditions are not sufficient: Pohoata~\cite{Pohoata14} constructed a weakly sparse family of graphs (later rediscovered by Davies~\cite{Davies22}, and now called \emph{Pohoata--Davies grids}; see \cref{fig:pohoata-grid}) that excludes a sufficiently large complete binary tree as an induced minor yet still has unbounded pathwidth (and treewidth).

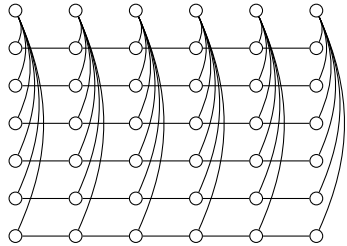
\begin{figure}[h!]
	\centering
	\begin{tikzpicture}[vertex/.style={draw,circle,inner sep=0.06cm}]
		\def\k{6}
		\pgfmathtruncatemacro\km{\k - 1}
		\def\sv{0.5}
		\def\sh{0.8}
		
		% vertices
		\foreach \i in {1,...,\k}{
			\foreach \j in {1,...,\k}{
				\node[vertex] (a\i\j) at (\i * \sh, \j *\sv) {};
			}
		}
		\foreach \i in {1,...,\k}{
			\node[vertex] (s\i) at (\i * \sh, \k * \sv + \sv) {};
		}
		
		% horizontal edges
		\foreach \i [count = \ip from 2] in {1,...,\km}{
			\foreach \j in {1,...,\k}{
				\draw (a\i\j) -- (a\ip\j) ;
			}
		}
		
		% star edges
		\foreach \i in {1,...,\k}{
			\foreach \j in {1,...,\k}{
				\draw (s\i) to [bend left = 24] (a\i\j) ;
			}
		}
	\end{tikzpicture}
	\caption{The $6 \times 6$ Pohoata--Davies grid.}
        \label{fig:pohoata-grid}
\end{figure}

In this paper, we determine precisely the graphs whose exclusion as an induced minor implies bounded pathwidth within every weakly sparse graph class.
Call $K_{1,3}$ a \emph{claw}.
An \emph{asteroidal triple} in a~graph is an independent set of three vertices such that between any two of them there exists a~path avoiding the neighborhood of the third.

\begin{theorem}\label{thm:main}
	For every graph $F$ and integer $t \geqslant 2$, the class of graphs excluding $F$ as an induced minor and $K_{t,t}$ as a~subgraph has bounded pathwidth if and only if $F$ is a forest \emph{without} either of the following:
	\begin{compactitem}
		\item three pairwise disjoint claws whose contraction into say, $a$, $b$, $c$, yields a~forest in which $a, b, c$ form an asteroidal triple; 
		\item four pairwise disjoint claws on vertex subsets $A, B, C, D$ such that there is an $A$--$B$ path $P$ in $F-N_F[C \cup D]$, a~$C$--$D$ path $Q$ in $F-N_F[A \cup B]$, at most one edge between $A \cup B$ and $C \cup D$, and if this edge exists, it is not incident to any claw center.
	\end{compactitem}
\end{theorem}

We say a~graph $H$ is \emph{\posi} if every weakly sparse class excluding $H$ as an induced minor has bounded pathwidth, and \emph{\nega} otherwise.
As we will see, \cref{thm:main} can be reformulated as follows.

\begin{theorem}[reformulation of~\cref{thm:main}]\label{thm:main-alt}
	A~graph is \posi if and only if it is an induced minor of $T_{2,\ell} \cup \ell T_{1,\ell}$ or of~$T_{3,\ell} \cup \ell T_{1,\ell}$ for some integer $\ell \geqslant 2$.
\end{theorem}

In the previous theorem, $T_{1,\ell}$ is the graph of~\cref{fig:T1l}, $T_{2,\ell}$, of~\cref{fig:T2l}, and $T_{3,\ell}$, of~\cref{fig:T3l}. 
The class $\mathcal F$ of forests mentioned in the abstract is simply the induced-minor closure of $\{T_{2,\ell} \cup \ell T_{1,\ell}, T_{3,\ell} \cup \ell T_{1,\ell}~:~\ell \geqslant 2\}$.
Thus, the \posi graphs are exactly the forests of~$\mathcal F$.
We present both formulations because, as will become clear, \cref{thm:main} is better suited for showing that a~graph is \nega, whereas \cref{thm:main-alt} is better suited for proving that a~forest is \posi.
Every graph that is not a~forest is \nega, since the class of forests has unbounded pathwidth.
Thus we \emph{only} need to classify forests.

We say that a~graph $H$ is \emph{$\tw$-\posi} if every weakly sparse class excluding $H$ as an induced minor has bounded treewidth, and \emph{$\tw$-\nega} otherwise.
Since bounded-treewidth graphs that exclude a~forest as an induced minor have bounded pathwidth~\cite{Hickingbotham22}, a forest is \posi if and only if it is $\tw$-\posi. 
Bonamy, Bonnet, Déprés, Esperet, Geniet, Hilaire, Thomassé, and Wesolek~\cite{Bonamy24} constructed a weakly sparse family of graphs (the so-called \emph{death stars}; see \cref{fig:death-star}) of unbounded treewidth that excludes $2K_3$ as an induced minor.
Ahn, Gollin, Huynh, and Kwon~\cite{AGHK2025EP} further showed that any graph $H$ with no $2K_3$ induced minor (which includes all $\tw$-\posi graphs) contains a~set $X\subseteq V(H)$ with $|X|\leq 12$ such that the subgraph of $H$ induced by deleting all the vertices at distance at~least~4 from $X$ is a forest.
Thus, every $\tw$-\posi graph is close to being a forest.
As such, classifying the $\tw$-\posi forests is a~significant step toward the broader task of classifying all $\tw$-\posi graphs. 

\medskip

Contrary to $H$-minor detection, which admits a~polynomial-time algorithm for any fixed~$H$, the complexity of $H$-induced-minor detection is less clear and offers a~richer landscape. 
Korhonen and Lokshtanov~\cite{KL2024InducedMinor} showed that there exist trees whose detection as an induced minor is NP-complete.
On the other hand, Dallard, Dumas, Hilaire, and Perez~\cite{Dallard25} give some infinite families of graphs $H$ such that $H$-induced-minor detection is polynomial-time solvable; we refer the interested reader to their introduction for a~detailed survey on this topic.

An algorithmic consequence of~\cref{thm:main-alt} is an efficient detection of \posi forests as induced minors within any weakly sparse class.  

\begin{corollary}\label{cor:alg}
	Let $\mathcal C$ be a~weakly sparse class and $F$ be an induced minor of some $T_{2,\ell} \cup \ell T_{1,\ell}$ or of~$T_{3,\ell} \cup \ell T_{1,\ell}$.
	Then, one can decide whether $F$ is an induced minor of an input $G \in \mathcal C$ in linear time.
\end{corollary}

Indeed, this follows from a~classical win-win argument.
First, compute a~2-approximation of the treewidth of the input graph $G$ using the fixed-parameter linear-time algorithm of~\cite{Korhonen21}, which outputs a~tree-decomposition. 
Either the treewidth of $G$ is sufficiently large and we can immediately conclude $G$ admits $F$ as an induced minor, or its treewidth is upper bounded (by a~function of $\mathcal C$ and~$F$) and we can invoke Courcelle's theorem~\cite{Courcelle90} as the existence of an $F$ induced-minor can be expressed in monadic second-order logic.
In case $F$ is an induced minor of~$G$, one further obtains in polynomial time an induced minor model of~$F$ in~$G$, by checking that the proofs in~\cite{istd16,istd18} and the current paper can be made effective.  

Let us mention some results related to~\cref{cor:alg} in that they address the complexity of induced-minor detection within ``sparse'' graph classes. 
By a~similar scheme and a~theorem due to Korhonen~\cite{Korhonen23}, one can efficiently detect any fixed planar induced minor in any bounded-degree class.
This can be lifted to any class excluding a~fixed topological minor~\cite{AbrishamiACHS24,Bonnet25}. 
Finally, detecting \emph{any} fixed graph as an induced minor is tractable in classes excluding a~fixed path as an induced subgraph~\cite{Dallard25}. 

\subsection{Proof Outline}

Our proof relies on the theory of constellations developed in the \emph{Induced Subgraphs and Tree Decompositions} series~\cite{istd-series}.
A~\emph{constellation} is a~graph that contains a spanning induced minor model of a~biclique where the branch sets for one side of the bipartition are singletons, and those on the other side induce paths.
There are two important types of constellations: \emph{interrupted} (generalizations of death stars) and \emph{zigzagged} (generalizations of Pohoata--Davies grids); see \cref{subsec:constellations} for their definitions. 

Chudnovsky, Hajebi, and Spirkl~\cite{istd18} showed that for any forest $F$, every weakly sparse $F$-induced-minor-free graph with sufficiently large pathwidth contains a~large interrupted or zigzagged constellation as an induced subgraph; see \cref{lem:pw-constell}.
Their work establishes the unavoidable (families of) induced subgraphs in classes of unbounded pathwidth.
Since large constellations have large pathwidth (as they contain large bicliques as induced minors), it follows that a~forest $F$ is \posi if and only if all large interrupted or zigzagged constellations contain $F$ as an induced minor. 

In \cref{sec:constructions} we show that every forest that meets one of the items in \cref{thm:main} is not an induced minor of either the Pohoata--Davies grids or the death stars.
As these graph families (or slight variants on them) are $K_{2,2}$-subgraph-free with unbounded treewidth; this yields the forward direction of \cref{thm:main}.
We use that to prove, in \cref{sec:negative-forests}, that every forest that is not an induced minor of any $T_{2,\ell} \cup \ell T_{1,\ell}$ or~$T_{3,\ell} \cup \ell T_{1,\ell}$ satisfies one of the items of~\cref{thm:main}, and hence is~\nega. 
Finally in \cref{sec:positive-forests}, we prove that every other forest (i.e., any induced minor of some $T_{2,\ell} \cup \ell T_{1,\ell}$ or~$T_{3,\ell} \cup \ell T_{1,\ell}$) is an induced minor of any sufficiently large interrupted or zigzagged constellation, and hence is \posi (by~\cref{lem:pw-constell}).
This completes the proof of \cref{thm:main-alt} and establishes its equivalence with \cref{thm:main}.

\section{Preliminaries}\label{sec:prelim}

Given integers~$i,j$, we denote by $[i,j]$ the set of integers that are at least $i$ and at~most~$j$, and $[i]$ is a~shorthand for $[1,i]$. 

\subsection{Standard graph-theoretic definitions and notation}

We denote by $V(G)$ and $E(G)$ the set of vertices and edges of a graph $G$, respectively.
A~graph $H$ is a~\emph{subgraph} of a~graph $G$, denoted by $H \subseteq G$, if $H$ can be obtained from $G$ by vertex and edge deletions.
Graph~$H$ is an~\emph{induced subgraph} of $G$ if $H$ is obtained from $G$ by vertex deletions only.
For $S \subseteq V(G)$, the \emph{subgraph of $G$ induced by $S$}, denoted $G[S]$, is obtained by removing from $G$ all the vertices that are not in $S$.
Then $G-S$ is a shorthand for $G[V(G)\setminus S]$.
We denote by $N_G(v)$ the set of neighbors of~$v$ in~$G$, define $N_G[v] := N_G(v) \cup \{v\}$, $N_G[S] := \bigcup_{v \in S} N_G[v]$, $N_G(S) := N_G[S] \setminus S$, and may omit the $G$ subscript when it is clear from the context. For a graph $J$ and integer $\ell\geq 1$, the graph $G\cup J$ is the disjoint union of $G$ and $J$ and the graph $\ell J$ is the disjoint union of $\ell$ copies of $J$.

A~set $X \subseteq V(G)$ is connected (in $G$) if $G[X]$ has a~single connected component.
The \emph{girth} of a~graph is the number of edges in one of its shortest cycles, and $\infty$ if the graph is acyclic.
A~graph class is \emph{weakly sparse} if it excludes $K_{t,t}$ as a~subgraph for some positive integer~$t$.
A~\emph{subdivision} of a~graph $G$ is any graph $H$ obtained from $G$ by replacing edges $e$ of~$G$ by paths with at~least~one edge whose extremities are the endpoints of~$e$.
We then say that $H$ contains $G$ as an \emph{induced subdivision}.

\medskip

A~\emph{tree-decomposition} of a~graph $G$ is a~family $(B_x)_{x\in V(T)}$ of subsets of~$V(G)$ (called \emph{bags}) indexed by the vertices of a tree $T$, such that
\begin{compactitem}
\item for every vertex $v \in V(G)$, $\{x \in V(T)~:~v \in B_x\}$ induces a non-empty (connected) subtree of $T$, and 
\item for every edge $uv \in E(G)$, some bag $B_x$ contains both $u$ and $v$.
\end{compactitem}
  \medskip
  The \emph{width} of~$(B_x)_{x\in V(T)}$ is $\max\{|B_x|~:~x \in V(T)\}-1$.
  The \emph{treewidth} of $G$ is the minimum width of a tree-decomposition of $G$. 
  A~\emph{path-decomposition} is a tree-decomposition in which the underlying tree is a path, and similarly the \emph{pathwidth} of~$G$ is the minimum width of a~path-decomposition of~$G$.

\medskip
  
  A~graph $H$ is a \emph{minor} of a~graph $G$ if $H$ is isomorphic to a~graph that can be obtained from a subgraph of $G$ by contracting edges.
  A~graph $H$ is an \emph{induced minor} of a~graph $G$ if $H$ is isomorphic to a~graph that can be obtained from an induced subgraph of $G$ by contracting edges.
  Equivalently, $H$ is an induced minor of $G$ if there is a~collection $\mathcal M=\{X_v \subseteq V(G)~:~v \in V(H)\}$ of pairwise disjoint connected sets (called \emph{branch sets}) such that $X_u$ and $X_v$ are adjacent if and only if $uv \in E(H)$.
  We call $\mathcal M$ an \emph{induced minor model} of~$H$ in~$G$.
  A~graph~$G$ is \emph{$H$-induced-minor-free} (resp.~\emph{$H$-minor-free}) if~$H$ is not an induced minor (resp.~a~minor) of~$G$.

\subsection{Constellations}\label{subsec:constellations}
  
Constellations are introduced in the \emph{Induced Subgraphs and Tree Decompositions} series~\cite{istd-series}; see for instance~\cite{istd16}.
We recall the relevant definitions and results.

A~\emph{constellation} $\cro = (S_\cro, \mathcal L_\cro)$ is~a~graph (also denoted by $\cro$) such that $S_\cro$ is an independent set of $\cro$, every connected component of $\cro-S_\cro$ is a~path, each of which is an element of~$\mathcal L_\cro$, and every vertex of $S_\cro$ has at~least one neighbor in each path of~$\mathcal L_\cro$.
To highlight the size of the independent set $S_\cro$ and the number of paths in $\mathcal L_\cro$, $\cro$ is more specifically called an \emph{$(|S_\cro|, |\mathcal L_\cro|)$-constellation}. 
In particular, $\{\{v\}~:~v \in S_\cro\} \cup \{V(P)~:~P \in \mathcal L_\cro\}$ is an induced minor model of the biclique $K_{|S_\cro|, |\mathcal L_\cro|}$.
Also note that for every $S \subseteq S_\cro$ and $\mathcal L \subseteq \mathcal L_\cro$, the pair $(S,\mathcal L)$ induces a~constellation in~$\cro$.

A~\emph{$\cro$-route} is a~path in~$\cro$ whose endpoints are both in~$S_\cro$ but no internal vertices of the path are in~$S_\cro$.
In particular, the internal vertices of a~$\cro$-route induce a~subpath of some element of~$\mathcal L_\cro$.
A~constellation $\cro$ is \emph{$d$-ample} if there is no $\cro$-route on at~most~$d+1$ edges. In particular, for $d\geq 1$, the neighborhoods of vertices in $S_\cro$ are pairwise disjoint.
For a~technical reason, we will need that the vertices close to endpoints of the paths of $\mathcal L_\cro$ have no neighbors in~$S_\cro$.
We now show that this can be arranged without loss of generality.

\begin{lemma}\label{lem:technical}
  For any positive integers $s, p, d$, and any $d$-ample $(s+2p,p)$-constellation $\cro = (S_{\cro}, \mathcal L_{\cro})$, there is some $S \subseteq S_\cro$ such that in the $d$-ample $(s,p)$-constellation induced by $(S, \mathcal L_\cro)$ no vertex of $S$ has a~neighbor in $P \in \mathcal L_\cro$ at~distance (in~$P$) less than~$d$ from an extremity of~$P$.
\end{lemma}

\begin{proof}  
  Let $\widehat{S}\subseteq S_\cro$ be the set of vertices that have a~neighbor closest to an end of a path in $\mathcal{L}_\cro$. As there are $p$ paths, $|\widehat{S}|\leq 2p$. Let $S\subseteq S_\cro\setminus\widehat{S}$ with $|S|=s$.
  Then $(S, \mathcal L_\cro)$ is a $d$-ample constellation.
  The property of the lemma then follows by the $d$-ampleness of~$\cro$.
\end{proof}

For simplicity, we henceforth incorporate the conclusion of \cref{lem:technical} into the definition of $d$-ampleness.
Due to this lemma, the subsequent results of the section hold with this slight strengthening. 

We will deal with two kinds of constellations: \emph{interrupted} and \emph{$q$-zigzagged} for some positive integer~$q$.
The former ones generalize the so-called \emph{death star} construction, while the latter extend the Pohoata--Davies grids (see~\cref{sec:constructions} for these constructions).
Both types of constellations entail a~total order over~$S_\cro$.
A~constellation $\cro = (S_\cro, \mathcal L_\cro)$ is \emph{interrupted} if there is a~total order $\prec$ on~$S_\cro$ such that for every $x \prec y \prec z$, every $\cro$-route between $x$ and $y$ is adjacent to~$z$.
For any positive integer $q$, $\cro$ is instead \emph{$q$-zigzagged} if there is a~total order $\prec$ on~$S_\cro$ such that for every $\cro$-route $P$ between $x$ and $y$, there are fewer than $q$ vertices of $\{z~\in S_\cro:~x \prec z \prec y\}$ that are non-adjacent to~$V(P)$.
Again, one can observe that for every $S \subseteq S_\cro$ and $\mathcal L \subseteq \mathcal L_\cro$, the pair $(S,\mathcal L)$ is an interrupted (resp.~$q$-zigzagged) constellation if $\cro$ is an interrupted (resp.~$q$-zigzagged) constellation.

The following is a~restatement of~\cite[Theorem 2.1]{istd18}.\footnote{The outcomes of item (a) in~\cite[Theorem 2.1]{istd18} are forbidden by the absence of $K_{t,t}$ subgraph and forest induced minor.}

\begin{lemma}[Theorem 2.1 in~\cite{istd18}]\label{lem:pw-constell}
  There are functions $f_{\ref{lem:pw-constell}}$ and $g_{\ref{lem:pw-constell}}$ such that for any forest $F$ and positive integers $t, s, p, d$, and $q \geqslant f_{\ref{lem:pw-constell}}(|V(F)|,t)$, any $K_{t,t}$-subgraph-free $F$-induced-minor-free graph with pathwidth at least $g_{\ref{lem:pw-constell}}(|V(F)|,t,s,p,d)$ admits as an induced subgraph a~$d$-ample $(s,p)$-constellation that is either interrupted or $q$-zigzagged.   
\end{lemma}

In other words, every weakly sparse class excluding some forest induced minor either has bounded pathwidth or admits as induced subgraphs arbitrarily large and ample constellations that are interrupted or $O(1)$-zigzagged. \cref{lem:pw-constell} is the key result used by Chudnovsky, Hajebi, and Spirkl~\cite{istd18} to describe the unavoidable induced subgraphs of graphs with large pathwidth.
Due to~\cref{lem:technical}, \cref{lem:pw-constell} holds with the slightly stronger notion of $d$-ampleness. 

\section{The Pohoata--Davies grid and the Death Star}\label{sec:constructions}

Given a~positive integer $n$, the $n \times n$ Pohoata--Davies grid is the graph obtained from the disjoint union of $n$ paths on $n$ vertices each, by adding for each $i \in [n]$, a~vertex adjacent to the $i$-th vertex of each path; see~\cref{fig:pohoata-grid}.
We call the vertices on the $n$ paths \emph{path vertices}, and the added vertices of degree~$n$ are the \emph{star vertices}.

\begin{lemma}\label{lem:npd}
  Let $H$ be any graph, and $A, B, C \subset V(H)$ be three non-empty sets such that for every $X \in \{A, B, C\}$,
  \begin{compactitem}
  \item $N_H[X]$ is disjoint from both sets in $\{A, B, C\} \setminus \{X\}$,
  \item $H[X]$ is connected and is \emph{not} a path, and
  \item the two sets in $\{A, B, C\} \setminus \{X\}$ are in the same connected component of~$H-N_H[X]$.
  \end{compactitem}
  Then $H$ is not an induced minor of any Pohoata--Davies grid.
\end{lemma}

\begin{proof}
  Suppose, for the sake of contradiction, that there is an induced minor model of $H$ in a~Pohoata--Davies grid.
  By the second item, for every $X \in \{A, B, C\}$ at least one branch set of an element of $X$ contains at least one \emph{star} vertex, say,~$s_X$.
  Then one of $s_A, s_B, s_C$ has a~neighborhood whose removal disconnects the other two of~$s_A, s_B, s_C$.
  This is incompatible with the first or the third item.
\end{proof}

\Cref{fig:npd} illustrates \cref{lem:npd} with two minimal trees whose exclusion as induced minors in weakly sparse classes does not yield bounded treewidth, as the family of Pohoata--Davies grids subsists.
Note that when applying this lemma, one can equivalently restrict $A, B, C$ to each induce a~claw or a cycle (the minimal connected non-path graphs).

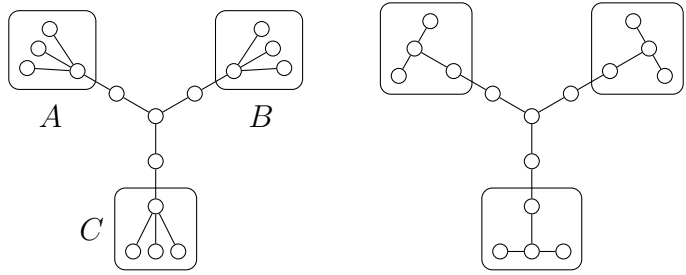
\begin{figure}[!ht]
  \centering
  \begin{tikzpicture}[vertex/.style={circle,draw,inner sep=2pt}]

  \begin{scope}
  \def\s{0.6}

  \node[vertex] (v) at (0,0) {} ;
  
  \foreach \a/\x in {30/1,150/2,270/3}{
    \begin{scope}[rotate=\a]
    \foreach \i/\j/\y in {1/0/1, 2/0/2, 3/0/3, 3/-0.5/4, 3/0.5/5}{
      \node[vertex] (v\x\y) at (\i * \s,\j * \s) {} ;
    }
    \draw (v) -- (v\x1) -- (v\x2) -- (v\x3) ;
    \draw (v\x4) -- (v\x2) -- (v\x5) ;
    \end{scope}
  }

  \node[draw, rounded corners, fit=(v12) (v13) (v14) (v15)] (A) {} ;
  \node[draw, rounded corners, fit=(v22) (v23) (v24) (v25)] (B) {} ;
  \node[draw, rounded corners, fit=(v32) (v33) (v34) (v35)] (C) {} ;

  \node at (-1.4,0) {$A$} ;
  \node at (1.4,0) {$B$} ;
  \node at (-0.85,-1.5) {$C$} ;
  \end{scope}

  \begin{scope}[xshift=5cm]
  \def\s{0.6}

  \node[vertex] (v) at (0,0) {} ;
  
  \foreach \a/\x in {30/1,150/2,270/3}{
    \begin{scope}[rotate=\a]
    \foreach \i/\j/\y in {1/0/1, 2/0/2, 3/0/3, 3/-0.7/4, 3/0.7/5}{
      \node[vertex] (v\x\y) at (\i * \s,\j * \s) {} ;
    }
    \draw (v) -- (v\x1) -- (v\x2) -- (v\x3) ;
    \draw (v\x4) -- (v\x3) -- (v\x5) ;
    \end{scope}
  }

  \node[draw, rounded corners, fit=(v12) (v13) (v14) (v15)] (A) {} ;
  \node[draw, rounded corners, fit=(v22) (v23) (v24) (v25)] (B) {} ;
  \node[draw, rounded corners, fit=(v32) (v33) (v34) (v35)] (C) {} ;
  \end{scope}

  \end{tikzpicture}
  \caption{Two minimal trees satisfying~\cref{lem:npd}, thus not being induced minors of any Pohoata--Davies grid.}
  \label{fig:npd}
\end{figure}

The death star of height $n$ is defined as follows.
The death star of height 1 is simply an edge, with one vertex forming the \emph{path} of the death star, and the other vertex being its unique \emph{star vertex}.
The death star of height $n$ is obtained from the death star of height $n-1$, by subdividing each edge of its path, adding a~vertex at each end, thereby forming the path of the next death star, and then adding an $n$-th star vertex adjacent to all the new vertices of the path; see~\cref{fig:death-star}.
Again, we call the vertices on the path of the death star \emph{path vertices}.

\begin{figure}[h!]
  \centering
\begin{tikzpicture}[scale=.4]
\def\k{5}
\pgfmathtruncatemacro\km{\k-1}
\def\mw{0.2}
\node[draw,circle,minimum width=\mw cm, inner sep=0.03cm] (u1) at (0,0) {} ;
%\node at (0,0.75) {\LARGE{$1$}} ;
\foreach \i in {2,...,\k}{
\pgfmathtruncatemacro\q{2^\i / 2}
\pgfmathtruncatemacro\qm{\q / 2 - 1/2}
\pgfmathtruncatemacro\kk{\k - \i + 1}
\pgfmathtruncatemacro\s{2^\kk}
\pgfmathtruncatemacro\l{\k - \i}
\pgfmathtruncatemacro\ss{2^\l - 1}
\foreach \j in {0,...,\qm}{
\node[draw,circle,minimum width=\mw cm, inner sep=0.03cm] (u\i-\j) at (- \j * \s - \ss - 1,0) {} ;
%\node at (- \j * \s - \ss - 1,0.75) {\LARGE{$\i$}} ;
\node[draw,circle,minimum width=\mw cm, inner sep=0.03cm] (v\i-\j) at (\j * \s + \ss + 1,0) {} ;
%\node at (\j * \s + \ss + 1,0.75) {\LARGE{$\i$}} ;
}
}
\foreach \i in {1,...,\k}{
  \node[draw,circle,minimum width=\mw cm, inner sep=0.03cm] (w\i) at (0,2.5 + 2 * \i) {} ;
}

\pgfmathtruncatemacro\q{2^\k / 2 - 1}
\foreach \j in {-\q,...,\q}{
  \node[circle,minimum width=\mw cm, inner sep=0.03cm] (a\j) at (\j,0) {} ;
}

%edges
\pgfmathtruncatemacro\qp{-2^\k / 2 + 2}
\pgfmathtruncatemacro\qm{\q - 1}
\foreach \j [count = \jp from \qp] in {-\q,...,\qm}{
  \draw (a\j) -- (a\jp) ;
}

\draw (w1) -- (u1) ;
\foreach \i in {2,...,\k}{
\pgfmathtruncatemacro\q{2^\i / 2}
\pgfmathtruncatemacro\qm{\q / 2 - 1/2}
\foreach \j in {0,...,\qm}{
\draw (w\i) to [bend left = -5 * \i] (u\i-\j) ;
\draw (w\i) to [bend right = -5 * \i] (v\i-\j) ;
}
}
\end{tikzpicture}
\caption{The death star of height 5.}
\label{fig:death-star}
\end{figure}
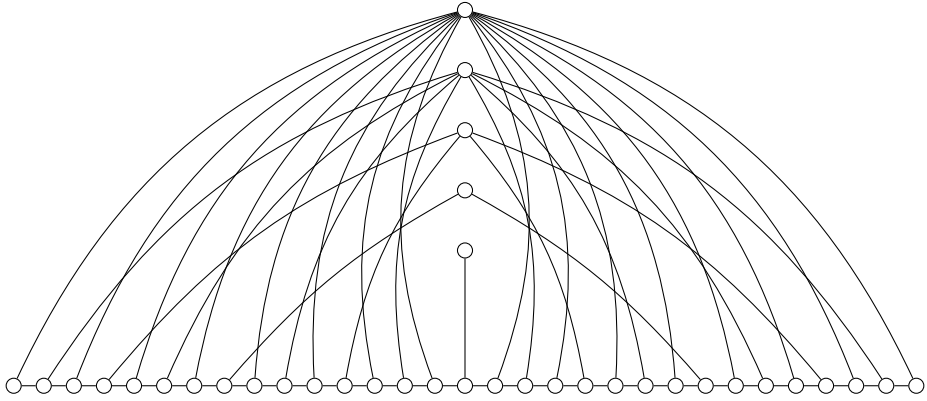

\begin{lemma}\label{lem:nds}
  Let $H$ be any graph, $A, B, C, D$ be four disjoint sets each inducing a~claw in~$H$, $P$ be an $A$--$B$ path, $Q$ be a~$C$--$D$ path such that
  \begin{compactitem}
  \item $V(P)$ and $N_H[C \cup D]$ are disjoint,
  \item $V(Q)$ and $N_H[A \cup B]$ are disjoint, and
  \item there is at most one edge between $A \cup B$ and $C \cup D$, and this edge (if it exists) is between a~leaf of the claws $A, B$ and a~leaf of the claws $C, D$.
  \end{compactitem}
  Then $H$ is not an induced minor of any death star.
\end{lemma}

\begin{proof}
  Suppose for the sake of contradiction that there is some integer $n$ such that the death star of height~$n$ admits an induced minor model of~$H$.
  Since for every $X \in \{A,B,C,D\}$, $H[X]$ is a~claw (which is not a~path), at least one of the four branch sets representing~$X$ has to contain a~star vertex $s_X$.
  Without loss of generality, we can assume that $A$ has the \emph{highest} (i.e., the one added last in the inductive construction) star vertex, which we now denote by $s_A$.
  The connected set $C \cup D \cup V(Q)$ in~$H$ yields some $s_C$--$s_D$ path in the death star.
  However, by design of the death star, there is no $s_C$--$s_D$ path that avoids the neighborhood of~$s_A$.
  The only way to reconcile this with the second and third items of the lemma is that $s_A$ is in a~leaf branch set $Y$ of the claw $A$ corresponding to $y \in V(H)$, and one of $s_C, s_D$, say without loss of generality, $s_C$, is in a~leaf branch set $Z$ of the claw $C$ corresponding to $z \in V(H)$, $yz$ is the unique edge of~$H$ between $A \cup B$ and $C \cup D$, and $V(Q) \cap Z = \emptyset$.

  Let $Z_1, Z_2, Z_3$ be the three branch sets of $C$ other than~$Z$, and let $Z_2$ correspond to the center of claw~$C$.
  There cannot be a~star vertex $s$ in $Z_1 \cup Z_2 \cup Z_3$, since otherwise there would be an $s$--$s_D$ path avoiding the neighborhood of~$s_A$.
  It means that $Z_1, Z_2, Z_3$ are three consecutive (disjoint) subpaths of the death star's path, with $Z_2$ being in between $Z_1$ and~$Z_3$.
  The only neighbors of $Z_2$ outside of $Z_1 \cup Z_3$ are star vertices, so to realize the edge between $Z_2$ and $Z$, there should be a~star vertex $s$ in $Z$ that is adjacent to $Z_2$ (be it, $s_C$ or some other star vertex).
  But then this again implies an $s$--$s_D$ path avoiding the neighborhood of~$s_A$ since the path $Q$ enters $C$ in $Z_1 \cup Z_2 \cup Z_3$ (note also that the star vertices form an independent set); a~contradiction.
\end{proof}

Note that we do not need to require that $V(P)$ and $V(Q)$ are disjoint or non-adjacent.
\Cref{fig:nds} illustrates \cref{lem:nds} with two minimal trees, one for which $A \cup B$ and $C \cup D$ are non-adjacent (left), and one where a~unique allowed edge links these two sets (right).

\begin{figure}[!ht]
  \centering
  \begin{tikzpicture}[vertex/.style={circle,draw,inner sep=2pt}]
    \def\s{0.6}
    \def\h{1.5}

    \begin{scope}
    \foreach \i in {1,...,5}{
      \node[vertex] (a\i) at (\i * \h, 0) {} ;
    }
    \draw (a1) -- (a2) -- (a3) -- (a4) -- (a5) ;

    \foreach \i in {1,2,4,5}{
      \foreach \j in {1,2,3}{
        \node[vertex] (b\i\j) at (\i * \h + 0.3 * \j * \h - 0.6 * \h, -\h) {} ;
        \draw (a\i) -- (b\i\j) ;
      }
    }

    \foreach \i in {1,2,4,5}{
      \node[draw,rounded corners,fit=(a\i) (b\i1) (b\i2) (b\i3)] {} ;
    }

    \foreach \i in {1,2,4,5}{
      \node[very thick,vertex] at (\i * \h, 0) {} ;
    }
    \draw[very thick] (a1) -- (a2) ;
    \draw[very thick] (a4) -- (a5) ;

    \node at (\h,-1.35 * \h) {$A$} ;
    \node at (2 * \h,-1.35 * \h) {$B$} ;
    \node at (4 * \h,-1.35 * \h) {$C$} ;
    \node at (5 * \h,-1.35 * \h) {$D$} ;
    
    \node at (1.5 * \h,0.35 * \h) {$P$} ;
    \node at (4.5 * \h,0.35 * \h) {$Q$} ;
    \end{scope}

    \begin{scope}[xshift=9cm]
    \foreach \i in {1,...,10}{
      \node[vertex] (a\i) at (\i * \s, 0) {} ;
    }
    \draw (a1) -- (a2) -- (a3) -- (a4) -- (a5) -- (a6) -- (a7) -- (a8) -- (a9) -- (a10);

    \foreach \i/\k in {1/0.5,2/1.5,4/4,7/7,10/9.5,11/10.5}{
      \node[vertex] (b\i) at (\k * \s, -\h) {} ;
    }
    
    \draw (b1) -- (a1) -- (b2) ;
    \draw (a4) -- (b4) ;
    \draw (a7) -- (b7) ;
    \draw (b10) -- (a10) -- (b11) ;

    \node[draw,rounded corners,fit=(a1) (a2) (b1) (b2)] {} ;
    \node[draw,rounded corners,fit=(a3) (a4) (a5) (b4)] {} ;
    \node[draw,rounded corners,fit=(a6) (a7) (a8) (b7)] {} ;
    \node[draw,rounded corners,fit=(a9) (a10) (b10) (b11)] {} ;

    \foreach \i in {2,3,8,9}{
      \node[very thick,vertex] at (\i * \s, 0) {} ;
    }
    \draw[very thick] (a2) -- (a3) ;
    \draw[very thick] (a8) -- (a9) ;

    \node at (1.2 * \s,-3.35 * \s) {$A$} ;
    \node at (4 * \s,-3.35 * \s) {$B$} ;
    \node at (7 * \s,-3.35 * \s) {$C$} ;
    \node at (9.8 * \s,-3.35 * \s) {$D$} ;

    \node at (2.5 * \s,0.85 * \s) {$P$} ;
    \node at (8.5 * \s,0.85 * \s) {$Q$} ;
    \end{scope}
    
  \end{tikzpicture}
  \caption{Two minimal trees satisfying~\cref{lem:nds}, thus not being induced minors of any death star.
    Left: no edge between $A \cup B$ and $C \cup D$.
    Right: one permitted (unique) edge between $A \cup B$ and $C \cup D$, namely between $B$ and $C$.
    It is indeed incident to two claw leaves outside $V(P) \cup V(Q)$.}
  \label{fig:nds}
\end{figure}
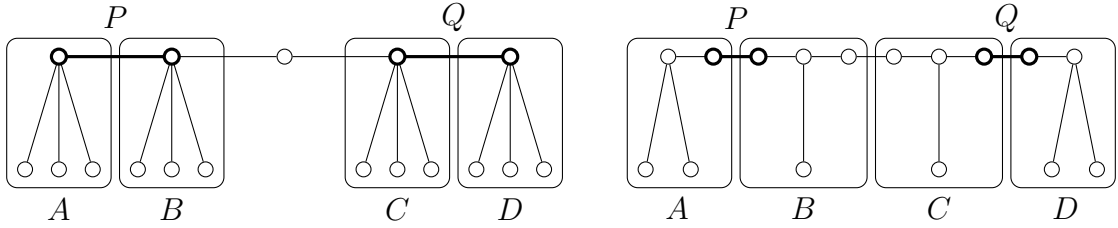

The families of Pohoata--Davies grids and death stars are weakly sparse: the former has girth~6 and the latter has no $K_{2,3}$ subgraph.
Note that the $n \times n$ Pohoata--Davies grid is a~$(n,n)$-constellation.
In particular, it admits the biclique $K_{n,n}$ as an (induced) minor, thus has treewidth at~least~$n$.
The death star of height $n$ has treewidth $\Theta(n)$~\cite{Bonamy24}.
Subdividing once every edge on the path of the death star of height~$n$, we get a~\emph{$K_{2,2}$}-subgraph-free construction with the same properties as those described in this section. 
Hence we get the following.

\begin{lemma}\label{lem:nega}
  Every graph that satisfies the preconditions of~\cref{lem:npd} or \cref{lem:nds} is \nega.
\end{lemma}

\section{Negative Forests}\label{sec:negative-forests}

In this section, we leverage \cref{lem:nega} to show that any forest that is not an induced minor of any $T_{2,\ell} \cup \ell T_{1,\ell}$ or $T_{3,\ell} \cup \ell T_{1,\ell}$ is negative.

Let $T_1$ be a~countably infinite tree defined as follows.
The root $r$ of $T_1$ has countably infinitely many children, each having exactly two children.
All the nodes of $V(T_1) \setminus N[r]$ have degree exactly 2: they have one parent and one child.
Let $\mathcal T_1$ be the set of finite connected induced subgraphs (or equivalently, induced minors) of $T_1$.
In particular, $\mathcal T_1$ is an infinite family of trees.

We also define a~family of finite trees containing every tree of~$T_1$ as induced subgraph.
For any positive integer $\ell$, let $T_{1,\ell}$ be the tree obtained from the disjoint union of $\ell$ paths on $2\ell$ edges, by adding a~vertex (the root) adjacent to the central vertex on every path; see~\cref{fig:T1l}.
The rooted tree $T_{1,\ell}$ is an important building block in what follows.

\begin{figure}[!ht]
  \centering
  \begin{tikzpicture}[vertex/.style={draw,circle,inner sep=0.06cm}]
    \def\s{1.3}
    \def\t{0.4}
    \def\l{5}
    \pgfmathsetmacro{\h}{1/\l)}
    \pgfmathtruncatemacro{\lm}{\l-1}
    \pgfmathtruncatemacro{\lmm}{\lm-1}
    
    \foreach \i/\j/\la in {0/0/r, -2/-1/z1,-1/-1/z2,0/-1/z3,1/-1/z4,2/-1/z5}{
      \node[vertex] (\la) at (\i * \s,\j * \s) {} ;
    }

    \foreach \i in {1,...,\l}{
      \draw (r) -- (z\i) ;
    }

    \foreach \i in {1,...,\l}{
      \node[vertex] (a\i) at (-3.35 * \s + \i * \s, -4) {} ;
      \node[vertex] (b\i) at (-2.65 * \s + \i * \s, -4) {} ;
      \foreach \p in {1,...,\lm}{
        \path (z\i) to node[vertex, pos=\p * \h] (a\i-\p) {} (a\i) ;
        \path (z\i) to node[vertex, pos=\p * \h] (b\i-\p) {} (b\i) ;
      }
      \draw (z\i) -- (a\i-1) ;
      \draw (z\i) -- (b\i-1) ;
      \foreach \p [count = \pp from 2] in {1,...,\lmm}{
        \draw (a\i-\p) -- (a\i-\pp) ;
        \draw (b\i-\p) -- (b\i-\pp) ;
      }
      \draw (a\i) -- (a\i-\lm) ;
      \draw (b\i) -- (b\i-\lm) ;
    }

    \begin{scope}[xshift=6.5cm, yshift=-1.5cm]
      \draw (0,0) -- (2 * \t,-1 * \t) -- (2.5 * \t,-3.5 * \t) -- (-2.5 * \t,-3.5 * \t) -- (-2 * \t, -1 * \t) -- cycle ;
      \node at (0,-2 * \t) {$T_{1,\ell}$} ;
    \end{scope}
  \end{tikzpicture}
  \caption{$T_{1,\ell}$ for $\ell=5$ (left) and its shorthand representation (right).}
  \label{fig:T1l}
\end{figure}
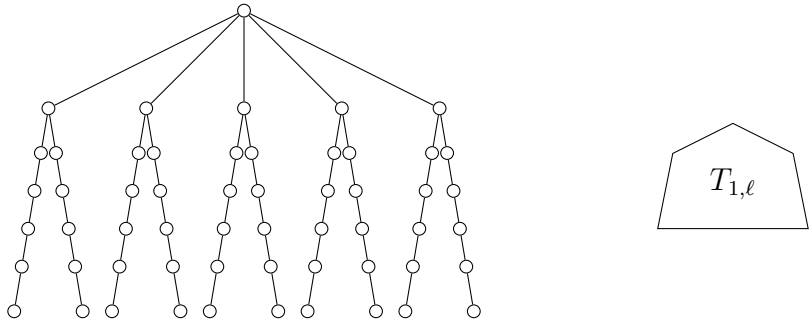

\begin{lemma}\label{lem:2-vd-claws}
  Any tree that is not in $\mathcal T_1$ admits two vertex-disjoint claws (linked by a~path).
\end{lemma}

\begin{proof}
  Note that there is a~path between any two vertex-disjoint sets in any tree, simply because trees are connected.
  Thus we only need to find two vertex-disjoint claws in any tree $T \notin \mathcal T_1$.

  Root $T$ at an arbitrary vertex.
  There should be a~claw in $T$, as otherwise $T$ is a~path, which contradicts $T \notin \mathcal T_1$.
  Now consider a~claw of~$T$ whose center is deepest in the rooted~$T$, and subject to that, pick if possible a~claw where all three leaves are children of its center.
  We denote its center by $u$, its three leaves by $u_1, u_2, u_3$, where $u_1$ and $u_2$ are children of~$u$ in~$T$, and $U := \{u,u_1,u_2,u_3\}$.
  There are two cases: $u_3$ is the parent of~$u$ or $u_3$ is a~child of~$u$ in~$T$.

  Let us analyze the former case.
  By how the claw is picked, it then holds that $u_1, u_2$ are the only two children of~$u$ in~$T$.
  All the descendants of $u_1, u_2$, themselves included, have at most one child.
  Otherwise there would be a~deeper claw.
  Now, if $T-\{u_3\}$ has a~connected component that is not a~path, we have found our two vertex-disjoint claws.
  But if all the connected components of $T-\{u_3\}$ are paths, then $T \in \mathcal T_1$, a~contradiction.

  We turn to the case when $u_3$ is a~child of~$u$ in the rooted~$T$.
  As $u$ is a~deepest claw center, the connected components of $T-\{u\}$ with descendants of~$u$ are all paths.
  The connected component of $T-\{u\}$ without descendants of~$u$ cannot be a~path (and has to exist) since otherwise $T$ would be in~$\mathcal T_1$.
  This yields a~claw not intersecting~$U$.
\end{proof}

\begin{lemma}\label{lem:2-not-in-T1}
  Let $F$ be a~forest with at least two connected components that are not in $\mathcal T_1$.
  Then $F$ is \nega.
\end{lemma}

\begin{proof}
  Let $T, T'$ be connected components of~$F$ that are not in~$\mathcal T_1$.
  By~\cref{lem:2-vd-claws}, there is an $A$--$B$ path $P$ in $T$ and a~$C$--$D$ path $Q$ in~$T'$ with four vertex-disjoint claws $A, B$ (in $T$) and $C,D$ (in $T'$).
  As $P$ and $C,D$ (resp.~$Q$ and $A,B$) are in distinct connected components of~$F$, the preconditions of~\cref{lem:nds} are met.
  We thus conclude by~\cref{lem:nega}.
\end{proof}

Fix an integer~$\ell \geqslant 2$.
We start by defining the tree $T'_{2,\ell}$.
First subdivide $\ell+1$ times two edges of a~star $K_{1,\ell+2}$; the center of the star is the root of $T'_{2,\ell}$.
Call the leaves in this tree \emph{original}.
Let $x$ and $y$ denote the non-leaf children of the root.	
From each parent of the two leaves non-adjacent to the root and from the child of $x$, we add an $\ell$-edge path (of descendants).
This finishes the construction of~$T'_{2,\ell}$.
It has $\ell+5$ leaves and $6\ell+5$ nodes.
The tree $T'_{3,\ell}$ is very similar to $T'_{2,\ell}$ (in particular, it has the same number of leaves and nodes):
The $\ell$-edge path from the child of $x$ is replaced by an $\ell$-edge path from $y$.

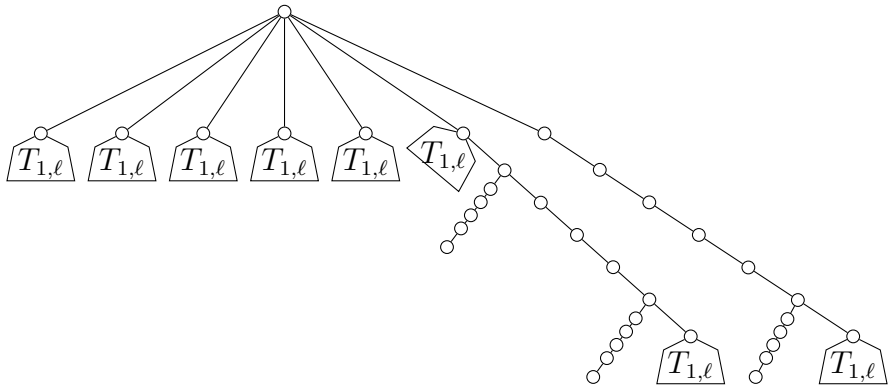
\begin{figure}[!ht]
  \centering
  \begin{tikzpicture}[vertex/.style={draw,circle,inner sep=0.06cm},scale=0.9]
    \def\s{1.2}
    \def\t{0.2}
    \def\l{5}
    \pgfmathsetmacro{\h}{1/(\l+1))}
    \pgfmathsetmacro{\hp}{1/\l)}
    \pgfmathtruncatemacro{\lm}{\l-1}
    \pgfmathtruncatemacro{\lmm}{\lm-1}
    
    \foreach \i/\j/\l in {0/0/r, -3/-1.5/z1,-2/-1.5/z2,-1/-1.5/z3,0/-1.5/z4,1/-1.5/z5, 2.2/-1.5/z6,3.2/-1.5/z7}{
      \node[vertex] (\l) at (\i * \s,\j * \s) {} ;
    }

    \foreach \i in {1,...,7}{
      \draw (r) -- (z\i) ;
    }

    \foreach \i in {1,...,5}{
      \draw (z\i) --++ (2 * \t,-1 * \t) --++ (0.5 * \t,-2.5 * \t) --++ (-5 * \t,0) --++ (0.5 * \t, 2.5 * \t) -- (z\i) ;
      \node at ([yshift=-0.45cm]z\i) {$T_{1,\ell}$};
    }

    \node[vertex] (x) at (5 * \s,-4 * \s) {} ;
    \foreach \p in {1,...,\l}{
        \path (z6) to node[vertex, pos=\p * \h] (a\p) {} (x) ;
    }
    \node[vertex] (y) at (7 * \s,-4 * \s) {} ;
    \foreach \p in {1,...,\l}{
        \path (z7) to node[vertex, pos=\p * \h] (b\p) {} (y) ;
    }

    \draw (z6) -- (a1) ;
    \draw (z7) -- (b1) ;
    \foreach \p [count = \pp from 2] in {1,...,\lm}{
      \draw (a\p) -- (a\pp) ;
      \draw (b\p) -- (b\pp) ;
    }
    \draw (a\l) -- (x) ;
    \draw (b\l) -- (y) ;

    \foreach \i in {x,y}{
      \draw (\i) --++ (2 * \t,-1 * \t) --++ (0.5 * \t,-2.5 * \t) --++ (-5 * \t,0) --++ (0.5 * \t, 2.5 * \t) -- (\i) ;
      \node at ([yshift=-0.45cm]\i) {$T_{1,\ell}$};
    }

    \begin{scope}[rotate=-40]
    \foreach \i in {z6}{
      \draw (\i) --++ (2 * \t,-1 * \t) --++ (0.5 * \t,-2.5 * \t) --++ (-5 * \t,0) --++ (0.5 * \t, 2.5 * \t) -- (\i) ;
      \node at ([yshift=-0.45cm]\i) {$T_{1,\ell}$};
    }
    \end{scope}

    \node[vertex] (xp) at (3.8 * \s, -4.5 * \s) {} ;
    \foreach \p in {1,...,\lm}{
        \path (a5) to node[vertex, pos=\p * \hp] (c\p) {} (xp) ;
    }

    \node[vertex] (yp) at (5.8 * \s, -4.5 * \s) {} ;
    \foreach \p in {1,...,\lm}{
        \path (b5) to node[vertex, pos=\p * \hp] (d\p) {} (yp) ;
    }

    \node[vertex] (zp) at (2 * \s, -2.9 * \s) {} ;
    \foreach \p in {1,...,\lm}{
        \path (a1) to node[vertex, pos=\p * \hp] (e\p) {} (zp) ;
    }

    \draw (a5) -- (c1) ;
    \draw (b5) -- (d1) ;
    \draw (a1) -- (e1) ;
    \foreach \p [count = \pp from 2] in {1,...,\lmm}{
      \draw (c\p) -- (c\pp) ;
      \draw (d\p) -- (d\pp) ;
      \draw (e\p) -- (e\pp) ;
    }
    \draw (c4) -- (xp) ;
    \draw (d4) -- (yp) ;
    \draw (e4) -- (zp) ;
  \end{tikzpicture}
  \caption{The tree $T_{2,\ell}$ for $\ell=5$ (and $T'_{2,\ell}$ by keeping only the root from each copy of $T_{1,\ell}$).}
  \label{fig:T2l}
\end{figure}

We can now define $T_{2,\ell}$ and $T_{3,\ell}$.
The tree $T_{2,\ell}$ (resp.~$T_{3,\ell}$) is obtained by adding $\ell$ children to $x$ and to each original leaf of~$T'_{2,\ell}$ (resp.~$T'_{3,\ell}$), and for each new vertex, adding a~$(2\ell+1)$-vertex path centered at this vertex.
In other words, we glue a~copy of $T_{1,\ell}$ at $x$ and each original leaf, identifying it with the root of the $T_{1,\ell}$; see~\cref{fig:T2l,fig:T3l}.

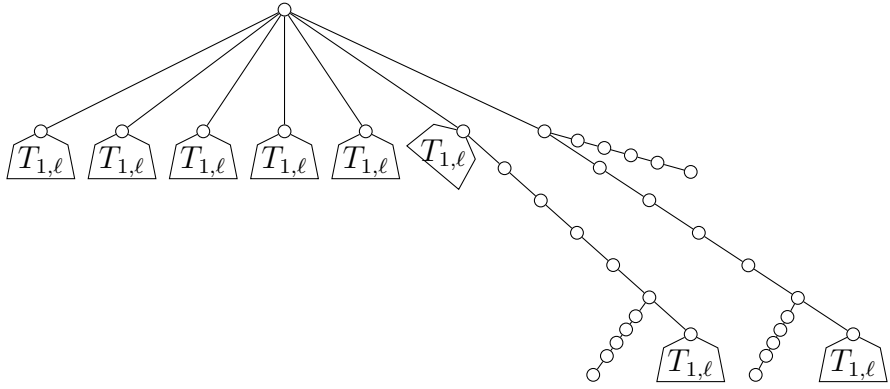
\begin{figure}[!ht]
  \centering
  \begin{tikzpicture}[vertex/.style={draw,circle,inner sep=0.06cm},scale=0.9]
    \def\s{1.2}
    \def\t{0.2}
    \def\l{5}
    \pgfmathsetmacro{\h}{1/(\l+1))}
    \pgfmathsetmacro{\hp}{1/\l)}
    \pgfmathtruncatemacro{\lm}{\l-1}
    \pgfmathtruncatemacro{\lmm}{\lm-1}
    
    \foreach \i/\j/\l in {0/0/r, -3/-1.5/z1,-2/-1.5/z2,-1/-1.5/z3,0/-1.5/z4,1/-1.5/z5, 2.2/-1.5/z6,3.2/-1.5/z7}{
      \node[vertex] (\l) at (\i * \s,\j * \s) {} ;
    }

    \foreach \i in {1,...,7}{
      \draw (r) -- (z\i) ;
    }

    \foreach \i in {1,...,5}{
      \draw (z\i) --++ (2 * \t,-1 * \t) --++ (0.5 * \t,-2.5 * \t) --++ (-5 * \t,0) --++ (0.5 * \t, 2.5 * \t) -- (z\i) ;
      \node at ([yshift=-0.45cm]z\i) {$T_{1,\ell}$};
    }

    \node[vertex] (x) at (5 * \s,-4 * \s) {} ;
    \foreach \p in {1,...,\l}{
        \path (z6) to node[vertex, pos=\p * \h] (a\p) {} (x) ;
    }
    \node[vertex] (y) at (7 * \s,-4 * \s) {} ;
    \foreach \p in {1,...,\l}{
        \path (z7) to node[vertex, pos=\p * \h] (b\p) {} (y) ;
    }

    \draw (z6) -- (a1) ;
    \draw (z7) -- (b1) ;
    \foreach \p [count = \pp from 2] in {1,...,\lm}{
      \draw (a\p) -- (a\pp) ;
      \draw (b\p) -- (b\pp) ;
    }
    \draw (a\l) -- (x) ;
    \draw (b\l) -- (y) ;

    \foreach \i in {x,y}{
      \draw (\i) --++ (2 * \t,-1 * \t) --++ (0.5 * \t,-2.5 * \t) --++ (-5 * \t,0) --++ (0.5 * \t, 2.5 * \t) -- (\i) ;
      \node at ([yshift=-0.45cm]\i) {$T_{1,\ell}$};
    }

    \begin{scope}[rotate=-40]
    \foreach \i in {z6}{
      \draw (\i) --++ (2 * \t,-1 * \t) --++ (0.5 * \t,-2.5 * \t) --++ (-5 * \t,0) --++ (0.5 * \t, 2.5 * \t) -- (\i) ;
      \node at ([yshift=-0.45cm]\i) {$T_{1,\ell}$};
    }
    \end{scope}

    \node[vertex] (xp) at (3.8 * \s, -4.5 * \s) {} ;
    \foreach \p in {1,...,\lm}{
        \path (a5) to node[vertex, pos=\p * \hp] (c\p) {} (xp) ;
    }

    \node[vertex] (yp) at (5.8 * \s, -4.5 * \s) {} ;
    \foreach \p in {1,...,\lm}{
        \path (b5) to node[vertex, pos=\p * \hp] (d\p) {} (yp) ;
    }

    \node[vertex] (zp) at (5 * \s, -2 * \s) {} ;
    \foreach \p in {1,...,\lm}{
        \path (z7) to node[vertex, pos=\p * \hp] (e\p) {} (zp) ;
    }

    \draw (a5) -- (c1) ;
    \draw (b5) -- (d1) ;
    \draw (z7) -- (e1) ;
    \foreach \p [count = \pp from 2] in {1,...,\lmm}{
      \draw (c\p) -- (c\pp) ;
      \draw (d\p) -- (d\pp) ;
      \draw (e\p) -- (e\pp) ;
    }
    \draw (c4) -- (xp) ;
    \draw (d4) -- (yp) ;
    \draw (e4) -- (zp) ;
  \end{tikzpicture}
  \caption{The tree $T_{3,\ell}$ for $\ell=5$ (and $T'_{3,\ell}$ by keeping only the root from each copy of $T_{1,\ell}$).}
  \label{fig:T3l}
\end{figure}

When $T$ is a~tree and $u, v \in V(T)$, we denote by $T[u,v]$ the unique $u$--$v$ path in~$T$.
When $T$ is a~rooted tree, we denote by $\lca(u,v)$ the least common ancestor of~$u$ and~$v$.
We also write $u \preceq_T v$ if $u$ is a~(not necessarily strict) ancestor of~$v$, and $u \prec_T v$ if $u \preceq_T v$ and $u \neq v$.

\begin{lemma}\label{lem:not-T2l}
  Every tree that does not satisfy the preconditions of~\cref{lem:npd} or of \cref{lem:nds} is an induced minor of~$T_{2,\ell}$ or of~$T_{3,\ell}$ for some integer $\ell$.
  Hence, every tree that is not an induced minor of $T_{2,\ell}$ or of $T_{3,\ell}$, for any $\ell$, is \nega.
\end{lemma}

\begin{proof}
  Let $T$ be any tree failing to meet the hypotheses of both \cref{lem:npd} and \cref{lem:nds}. Fix $\ell:= |V(T)|$.
  Let us root $T$ at some arbitrary vertex $u \in V(T)$. For each $x\in V(T)$, we denote by $T_x$ the subtree of~$T$ rooted at~$x$.
  Let $X \subseteq V(T)$ be the set of vertices $x$ such that $T_x$ contains a~claw but no proper rooted subtree of~$T_x$ contains a~claw.
  If $X$ is empty, then $T$ is a~path (hence an induced minor of~$T_{2,\ell}$).

  Let us first assume that $X=\{x\}$.
  Then it holds that every claw of~$T$ intersects $V(T[u,x])$ (at a~subpath).
  Let $A$ be any subset of $V(T_x)$ inducing a~claw.
  Let $D$ be a~subset of $V(T)$ inducing a~claw that is closest to the root in $T$, i.e., there is no $D' \subseteq V(T)$ inducing a~claw such that $D \cap V(T[u,x])$ contains a~descendant of a~vertex $D' \cap V(T[u,x])$ that is not itself in~$D' \cap V(T[u,x])$.
  We assume that $A$ and $D$ are disjoint, as otherwise $T$ is easily seen to be an induced minor of~$T_{2,\ell}$.

  Let $C$ be the vertex set of a~claw in~$T$ that is closest to the root among those disjoint from $D$.
  Again if $A \cap C \neq \emptyset$, $T$ is readily an induced minor of~$T_{2,\ell}$.
  Finally, let $B$ be the vertex set of a~claw in~$T$ that is furthest from the root among those disjoint from $A$.
  If $B$ and $C$ are not disjoint, then $T$ is an induced minor of the graph obtained from a~$(2\ell+1)$-vertex path with a~copy of $T_{1,\ell}$ rooted at the first, middle, and last vertex, and an $\ell$-vertex path attached to the second, penultimate, $\ell$-th and $(\ell+2)$-nd vertex.
  (Observe that the latter two vertices are the neighbors of the middle vertex.)
  Such a~tree is an induced minor of~$T_{2,\ell}$: solely keep an $\ell$-vertex path from the copies of $T_{1,\ell}$ adjacent to the root.
  As the preconditions of~\cref{lem:nds} are not met, $C$ and $D$ are adjacent, and if the centers of the claws $C,D$ are both in $V(T[u,x])$, they are not at distance~3 from each other.
  Therefore, $T$ is an induced minor of the graph $S_\ell$ obtained from a~$(2\ell+1)$-vertex path with a~copy of $T_{1,\ell}$ rooted at the first, $(\ell+1)$-st, $(\ell+2)$-nd, last vertex, and an $\ell$-vertex path attached to the second, penultimate, $\ell$-th vertex.
  This is itself an induced minor of $T_{2,\ell}$: contract the edge between the root and the first copy of $T_{1,\ell}$ and discard the $\ell-1$ other copies adjacent to the root.
  
  Henceforth, we assume that $|X| \geqslant 2$.
  Let $Z := \{\lca(x,y)~:~x \neq y \in X\}$.
  Observe that $|Z| \geqslant 1$ (since $|X| \geqslant 2$).
  The next claim establishes that $|Z| \leqslant 2$.

  \begin{claim}\label{clm:no-incomparable-lca}
    There are at most two distinct least common ancestors of distinct pairs of~$X$, and they are ancestor--descendants. 
  \end{claim}

  \begin{proof}
    First assume that $z = \lca(x,x')$ with $x \neq x' \in X$, $z'=\lca(y,y')$ with $y \neq y' \in X$, and $z$ and $z'$ are not ancestor--descendants.
    Let $A$ (resp.~$B$, $C$, $D$) be a~subset of $V(T_x)$ (resp.~$V(T_{x'})$, $V(T_y)$, $V(T_{y'})$) containing $x$ (resp.~$x'$, $y$, $y'$) and inducing a~claw.
    Then, $A,B,C,D$ and the paths $T[x,x'], T[y,y']$ satisfy~\cref{lem:nds}, contradicting our assumption for~$T$.

    Now assume that each of $z_1 \prec_T z_2 \prec_T z_3$ is a~least common ancestor of a~distinct pair of~$X$.
    Let $x \neq x' \in X$ be such that $z_3 = \lca(x,x')$.
    Then, let $y \in X$ be such that $z_2 = \lca(x,y)$, and $y' \in X$ be such that $z_1 = \lca(x,y')$.
    We then conclude as in the previous paragraph since $V(T[x,x'])$ is non-adjacent to $V(T_y) \cup V(T_{y'})$, and $V(T[y,y'])$ is non-adjacent to $V(T_x) \cup V(T_{x'})$.
  \end{proof} 

  \textbf{Case 1.} We now deal with the case when $|Z|=1$, say, $Z=\{z\}$.
  Let $W$ be the vertices of $X$ non-adjacent to $z$.
  As $T$ does not meet the conditions of~\cref{lem:npd}, $|W| \leqslant 2$.

  \textbf{Case 1.1.} We examine the subcase when $|W|=2$, say, $W=\{w_1,w_2\}$.
  Then, the component of $T-z$ containing the parent $p_z$ of~$z$ (if it exists) is a~connected induced subgraph of~$T_{1,\ell}$ rooted at~$p_z$, as otherwise the conditions of \cref{lem:nds} are satisfied.
  Let $z_1$ (resp.~$z_2$) be the neighbor of $z$ on~$T[z,w_1]$ (resp.~$T[z,w_2]$).
  Re-rooting $T$ at $z$, we have that, outside of $z_1, z_2$, the children of $z$ are roots of connected induced subgraphs of~$T_{1,\ell}$.
  Let $w'_1$ (resp.~$w'_2$) be the parent of~$w_1$ (resp.~$w_2$).
  Let $C_1$ (resp.~$C_2$) be the connected component of $T-\{z,w_1\}$ (resp.~$T-\{z,w_2\}$) containing $z_1$ (resp.~$z_2$).
  At~most one of $C_1,C_2$ contains a~claw, otherwise $T$ satisfies~\cref{lem:nds}.
  Say, without loss of generality that $C_2$ has no claw (hence is a~path).

  If $z_2$ has two children, then every internal vertex of $T[z_1,w'_1], T[z_2,w'_2]$ has degree~2, otherwise $T$ satisfies~\cref{lem:nds}.
  And it can then be observed that $T$ is an induced minor of some~$T_{3,\ell}$.
  If $z_2$ has only one child and $z$ has degree at~least~3, then $T$ is an induced minor of some~$T_{2,\ell}$.
  If finally $z_2$ has a~single child and $z$ has degree 2, then $T$ is an induced minor of~$S_\ell$ (as defined above), hence of~$T_{2,\ell}$.

  \textbf{Case 1.2.}
  We now assume that $|W|=1$, say $W=\{w\}$.
  Let $z'$ be the neighbor of $z$ on $T[z,w]$, and $w'$ the parent of~$w$.
  We re-root $T$ at~$z$.
  If among the new descendants of~$z$, there is a~claw that is non-adjacent to $z$, then we are back to Case 1.1.
  Indeed there can only be one new element of~$X$, as otherwise \cref{lem:nds} is satisfied.
  We can thus assume that every child of~$z$ but $z'$ is the root of a~connected induced subgraphs of~$T_{1,\ell}$.
  If $z'=w'$, the subtree of $T-\{w\}$ rooted at~$z'$ is an induced subgraph of $T_{1,\ell}$ rooted at~$z'$, thus $T$ is an induced minor of $T_{2,\ell}$ (and $T_{3,\ell}$).
  If $z'$ is the parent of~$w'$, in addition to the branch leading to~$w$, attached to $T[z',w']$ can only be: an induced subgraph of~$T_{1,\ell}$ rooted at~$z'$, and two paths of descendants for~$w'$; hence $T$ is an induced minor of $T_{2,\ell}$.
  If finally $z'$ and $w'$ are non-adjacent, in addition to the branch leading to~$w$, attached to $T[z',w']$ can only be: one path of descendants for~$w'$, one path of descendants for the neighbor of $z'$ on $T[z',w']$, and an induced subgraph of $T_{1,\ell}$ rooted at~$z'$; hence $T$ is an induced minor of $T_{2,\ell}$.
  It is indeed easy to check that anything else contradicts that there is a~single element of~$X$ in the subtree rooted at~$z'$, or the fact that $T$ does not satisfy the preconditions of~\cref{lem:nds}.

  \textbf{Case 1.3.}
  We finally assume that $W$ is empty.
  Then there are $x_1 \neq x_2 \in X$ both adjacent to~$z$.
  We re-root $T$ at~$z$, and recompute the new set~$X$.
  It holds that the new $X$ is equal to the old $X$ plus at most one new descendant of~$z$, as otherwise $T$ satisfies~\cref{lem:nds}.
  And this case can now be handled as Case 1.2.

  \medskip

  \textbf{Case 2.}
  We finally move to the case when $|Z|=2$, say, $Z=\{z_1,z_2\}$ and $z_1 \prec_T z_2$.
  Let $x, x' \in X$ be such that $z_2 = \lca(x,x')$ and $x'' \in X$ be such that $z_1=\lca(x,x'')$ (and equivalently $z_1=\lca(x',x'')$).
  One of $x, x'$ is adjacent to $z_2$, otherwise $T$ satisfies~\cref{lem:npd}.
  Say without loss of generality that $x$ and $z_2$ are adjacent.
  We re-root $T$ at~$z_2$.
  Among the new descendants of $z_2$, there is exactly one vertex in the new set~$X$, otherwise $T$ satisfies~\cref{lem:nds}.
  Hence we are back to Case 1.1.
\end{proof}

A~forest is not an induced minor of any $T_{2,\ell} \cup \ell T_{1,\ell}$ or $T_{3,\ell} \cup \ell T_{1,\ell}$ if and only if it has two connected components not in~$\mathcal T_1$ or one connected component that is not an induced minor of any $T_{2,\ell}$ or $T_{3,\ell}$.
Thus \cref{lem:2-not-in-T1,lem:not-T2l} imply the following.

\begin{lemma}\label{lem:nega2}
  Every forest that is not an induced minor of any $T_{2,\ell} \cup \ell T_{1,\ell}$ or $T_{3,\ell} \cup \ell T_{1,\ell}$ with $\ell \geqslant 2$ is \nega.
\end{lemma}

\section{Positive Forests}\label{sec:positive-forests}

In this section, we show that every forest that was not classified as \nega in~\cref{sec:negative-forests} is \posi.
We start with a~lemma that helps growing copies of $T_{1,\ell}$ in an induced minor model within an ample constellation.

\begin{lemma}\label{lem:induced-minor-extension}
  Let $\cro= (S_\cro, \mathcal L_\cro)$ be an $(\ell+1)$-ample constellation, and let $\mathcal M$ be an induced minor model of some graph $H$ in~$\cro$ such that there is a~path $P \in \mathcal L_\cro$ not intersected by~$\mathcal M$.
  Let $z \in S_\cro$ belong to the branch set of~$\mathcal M$ representing some~$u \in V(H)$.
  
  Then, there is an induced minor model $\mathcal M' \supset \mathcal M$ of $H'$ in $\cro$ with $\bigcup \mathcal M' \subseteq \bigcup \mathcal M \cup V(P)$, where $H'$ is the graph obtained from $H$ by adding a~disjoint $(2\ell+1)$-vertex path and making $u$ adjacent to the center of this path.
\end{lemma}

\begin{proof}
  Fix a~(non-empty) maximal subpath $P'$ of $P$ such that the first and last vertices of~$P'$ are adjacent to~$z$, and no vertex of $P'$ is adjacent to a~vertex in~$S_\cro \setminus \{z\}$.
  Note that $P'$ may be a~1-vertex path.
  Let $a_1, \ldots, a_\ell \in V(P)$ be the $\ell$ (distinct) vertices of~$P$ just before $P'$, and $b_1, \ldots, b_\ell \in V(P)$ be the $\ell$ (distinct) vertices of~$P$ just after $P'$.
  By~$(\ell+1)$-ampleness (in the strengthened sense) these vertices exist.
  We simply set $\mathcal M' := \mathcal M \cup \{\{a_1\}, \ldots, \{a_\ell\}, V(P'), \{b_1\}, \ldots, \{b_\ell\}\}$.

  By definition of~$P'$, there is no edge between $V(P')$ and $S_\cro \setminus \{z\}$.
  Since $\cro$ is $(\ell+1)$-ample, there is no edge between $\{a_1, \ldots, a_\ell, b_1, \ldots, b_\ell\}$ and $S_\cro \setminus \{z\}$.
  By maximality of~$P'$, there is no edge between $\{a_1, \ldots, a_\ell, b_1, \ldots, b_\ell\}$ and $\{z\}$.
  This guarantees that $\mathcal M'$ is an induced minor model of~$H'$. 
\end{proof}

We will need the following notion.
Given a~1-ample constellation $\cro = (S_\cro, \mathcal L_\cro)$, and two adjacent vertices $u \in S_\cro$ and $v$ on path $P \in \mathcal L_\cro$, the \emph{interval of $v$} is the vertex set of the maximal subpath $P'$ of $P$ subject to (1) $v \in V(P')$, (2) the endpoints of $P'$ are adjacent to $u$, and (3) no vertex of $S_\cro \setminus \{u\}$ is adjacent to~$V(P')$.

\begin{lemma}\label{lem:t2l-in-zigzagged}
  For any integers $\ell > q \geqslant 1$, any $q$-zigzagged $(\ell+1)$-ample $(2(q+2)(q+\ell+2),$ $(2\ell+3)\ell+2)$-constellation contains $T_{2,\ell} \cup \ell T_{1,\ell}$ and $T_{3,\ell} \cup \ell T_{1,\ell}$ as induced minors. 
\end{lemma}

\begin{proof}
  Let $\cro = (S_\cro, \mathcal L_\cro)$ be a~$q$-zigzagged $(\ell+1)$-ample $(2(q+2)(q+\ell+2),$ \mbox{$(2\ell+3)\ell+2)$}-constellation with total order $\prec$ on $S_\cro$.
  We will prove both claims of the lemma by constructing in $\cro$ an induced minor model of $T_{4,\ell} \cup \ell T_{1,\ell}$, where $T_{4,\ell}$ is the tree obtained by adding to $T_{2,\ell}$ the $\ell$-edge path that $T_{3,\ell}$ has but $T_{2,\ell}$ has not.
  This is sufficient since $T_{2,\ell}$ and $T_{3,\ell}$ are induced minors of~$T_{4,\ell}$.
  We denote by $T'_{4,\ell}$ the tree obtained by only keeping the root from each copy of $T_{1,\ell}$ in $T_{4,\ell}$.
  We will describe an induced minor model (of an induced subdivision) of $T'_{4,\ell} \cup \ell K_1$, and conclude by applying~\cref{lem:induced-minor-extension}.
  
  We call a~$\cro$-route that is not adjacent to any other vertices of $S_\cro$ than its endpoints, and inclusionwise minimal (i.e., the path is induced) a~\emph{direct} $\cro$-route.
  We first find $q+\ell+1$ direct $\cro$-routes with pairwise distinct endpoints.
  Let $s_1 \prec s_2 \prec \ldots$ enumerate $S_\cro$.
  For every $i \in [q+\ell+1]$, let $P_i$ be an arbitrarily chosen direct $\cro$-route starting at $s_{2(q+1)i}$ along some fixed path $Q$ of $\mathcal L_\cro$.
  Let $h_i$ be such that $P_i$ is an $s_{2(q+1)i}$--$s_{h_i}$ path.
  As $\cro$ is $q$-zigzagged, $h_i \in [2(q+1)i-q,2(q+1)i+q] \setminus \{2(q+1)i\}$.
  In particular, the paths $P_i$ are pairwise vertex-disjoint and non-adjacent.
   
  Let $P$ be a~shortest $\cro$-route between $x := s_1$ and $\{s_{2(q+1)(q+\ell+2)}, \ldots, s_{2(q+1)(q+\ell+2)+q}\}$ using some $Q' \in \mathcal L_\cro \setminus \{Q\}$.
  We denote by $y$ the other endpoint of $P$ (than~$x$).
  As $\cro$ is \mbox{$q$-zigzagged}, there are $\ell+2$ distinct vertices $z_1 \prec z_2 \prec \ldots \prec z_{\ell+2} \in S_\cro$ with each $z_j$ equal to some $s_{2(q+1)i}$ such that both $s_{2(q+1)i}$ and $s_{h_i}$, which we denote by $z_{j+\ell+1}$, are adjacent to~$V(P)$.
  Another consequence of~$q$-zigzaggedness (as $P$ is shortest) is that no vertex $s_j \in S_\cro$ with $j > 2(q+1)(q+\ell+2)+q$ is adjacent to $V(P)$.
  We set $\eta := 2(q+1)(q+\ell+2)+q$, and will eventually take the $\ell$ vertices $s_{\eta+1}, \ldots, s_{\eta+\ell}$ to realize $\ell K_1$.

  Let $I_x$ (resp.~$I_y$) be the interval of the neighbor of $x$ (resp.~of~$y$) on~$P$.
  Let $u_1, \ldots, u_\ell$ (resp.~$u'_1, \ldots, u'_\ell$) be the vertices of $Q'$ just before (resp.~just after)~$I_x$.
  Let $v_1, \ldots, v_\ell$ (resp.~$v'_1, \ldots, v'_\ell$) be the vertices of $Q'$ just after (resp.~just before)~$I_y$.
  By~$(\ell+1)$-ampleness (in the strengthened sense) $u_1, \ldots, u_\ell$ and $v_1, \ldots, v_\ell$ do exist.
  
  When traversing $P$ from $x$ to $y$, let $x' \in V(P)$ (resp.~$y' \in V(P)$) be the first (resp.~last) vertex with a~neighbor in~$\{z_1, \ldots, z_{2(\ell+1)}\}$.
  Let $P_{i(x')}$ (resp~$P_{i(y')}$) be the unique path in $\{P_1, \ldots, P_{\ell+1}\}$ with an endpoint that is a~neighbor of~$x'$ (resp.~of~$y'$).
  It holds that $i(x') \neq i(y')$ since $\cro$ is $q$-zigzagged and $\ell > q$.
  Indeed, we get $\cro$-routes between $x$ and $z_{i(x')}$ or $z_{i(x')+\ell}$, and between $z_{i(y')}$ or $z_{i(y')+\ell}$ and $y$ that are not adjacent to any vertices of $S_\cro$ other than their endpoints, and there are more than $2q$ vertices of $S_\cro$ between $x$ and $y$ along $\prec$. 
  
  Let $I_{x'}$ (resp.~$I_{y'}$) be the interval of~$x'$ (resp.~of~$y'$).
  Let $w_1, \ldots, w_\ell$ (resp.~$w'_1, \ldots, w'_\ell$) be the vertices of~$P$ just after $I_{x'}$ (resp.~just before~$I_{y'}$).
  Let $J$ be the set of vertices of~$P$ that are strictly between $\{u'_1, \ldots, u'_\ell\}$ and $I_{x'}$.
  Let $J'$ be the set of vertices of~$P$ that are strictly between $I_{y'}$ and $\{v'_1, \ldots, v'_\ell\}$.
  
  As $\cro$ is $(\ell+1)$-ample, the sets $\{u_1, \ldots, u_\ell\}$, $I_x$, $\{u'_1, \ldots, u'_\ell\}$, $I_{x'}$, $\{w_1, \ldots, w_\ell\}$, $\{w'_1, \ldots, w'_\ell\}$, $I_{y'}$, $\{v'_1, \ldots, v'_\ell\}$, $I_{y}$, $\{v_1, \ldots, v_\ell\}$ are pairwise disjoint.
  Let $R$ be the vertex set of a~minimal subpath of $P$ between $\{w_1, \ldots, w_\ell\}$ and $\{w'_1, \ldots, w'_\ell\}$ adjacent to $V(P_i)$ for every $i \in [\ell+2]$.
  There is such a~subpath since $i(x') \neq i(y')$, the subpath of~$P$ up to $\{w_1, \ldots, w_\ell\}$ is only adjacent to exactly one of $z_{i(x')}, z_{i(x')+\ell}$ and the subpath of $P$ after $\{w'_1, \ldots, w'_\ell\}$ is only adjacent to exactly one of $z_{i(y')}, z_{i(y')+\ell}$, while $P$ is adjacent to every vertex in~$\{z_1, \ldots, z_{2(\ell+2)}\}$. 
  By $(\ell+1)$-ampleness, $R$ is non-adjacent to $\{w_1, \ldots, w_\ell, w'_1, \ldots, w'_\ell\}$.
  \Cref{fig:constr-zigzagged} illustrates the vertices and subsets defined.

  \begin{figure}[!ht]
  \centering
  \begin{tikzpicture}[vertex/.style={draw,circle,inner sep=0.1cm}]
    \foreach \i/\j/\l in {1/0/x, 2/2/z_1,3/3.1/z_{\ell+1}, 4/4.2/z_2,5/4.9/z_{\ell+2}, 6/6/{z_{i(x')}},7/7/{z_{i(y')}}, 8/8.1/z_{\ell+2},9/9.3/z_{2\ell+2}, 10/11/y}{
      \node[vertex] (s\i) at (\j,0) {} ;
      \node at (\j,0.35) {$\l$} ;
    }

    \draw (s2) --++ (0.2,-1) --++(0.8,0) node[above,midway] {$P_1$} -- (s3) ;
    \draw (s4) --++ (0.1,-1) --++(0.5,0) node[above,midway] {$P_2$} -- (s5) ;
    \draw (s8) --++ (0.3,-1) --++(0.75,0) node[above,midway] {$P_{\ell+2}$} -- (s9) ;

    \node at (-1.5,-1) {$Q$} ;
    \draw[gray,very thin] (-1,-1) -- (12.5,-1) ;
    
    \node at (-1.5,-2) {$Q'$} ;
    \draw[gray,very thin] (-1,-2) -- (12.5,-2) ;

    \draw (s1) -- (0.6,-2) ;
    \draw (s6) -- (4.1,-2) ;
    \draw (s7) -- (9.1,-2) ;
    \draw (s10) -- (10.85,-2) ;

    \draw[->] (0.6,-2.8) to node[below,midway] {$P$} (10.85,-2.8) ;

    \draw[thick] (0.35,-2) --++(0.25,0) node[midway,below] {$I_x$};
    \draw[thick] (4.1,-2) --++(0.7,0) node[midway,below] {$I_{x'}$};
    \draw[thick] (8.81,-2) --++(0.29,0) node[midway,below] {$I_{y'}$};
    \draw[thick] (10.85,-2) --++(0.45,0) node[midway,below] {$I_y$};
    \draw[thick] (6.3,-2) --++(0.7,0) node[midway,below] {$R$};

    \node at (-0.35,-2.3) {\tiny{$u_\ell, \ldots, u_1$}} ;
    \node at (1.4,-2.3) {\tiny{$u'_1, \ldots, u'_\ell$}} ;
    \node at (3,-2.3) {$J$} ;

    \node at (5.5,-2.3) {\tiny{$w_1, \ldots, w_\ell$}} ;
    \node at (8.1,-2.3) {\tiny{$w'_\ell, \ldots, w'_1$}} ;
    \node at (10.25,-2.3) {\tiny{$v'_\ell, \ldots, v'_1$}} ;
    \node at (11.85,-2.3) {\tiny{$v_1, \ldots, v_\ell$}} ;
    \node at (9.4,-2.3) {$J'$} ;
    
  \end{tikzpicture}
  \caption{The vertices and subsets used in the subsequent induced minor model of~$T'_{4,\ell}$.}
  \label{fig:constr-zigzagged}
  \end{figure}
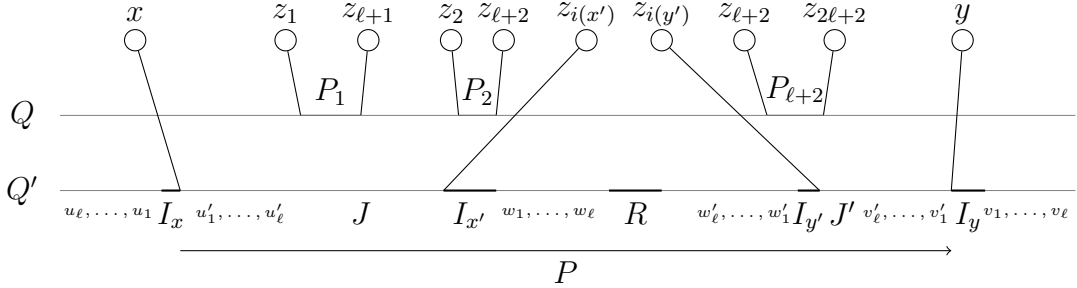

  We obtain an induced minor model $\mathcal M$ of an induced subdivision of~$T'_{4,\ell} \cup \ell K_1$ as follows:
  \[\mathcal M := \{\{x\}, \{u_1\}, \ldots, \{u_\ell\}, I_x, \{u'_1\}, \ldots, \{u'_\ell\}, J, \{w_1\}, \ldots, \{w_\ell\}, I_{x'}, V(P_{i(x')}),  R,\]
  \[V(P_1), \ldots, V(P_{i(x')-1}), V(P_{i(x')+1}), \ldots, V(P_{i(y')-1}), V(P_{i(y')+1}), \ldots, V(P_{\ell+2}), \{w'_1\}, \ldots, \{w'_\ell\},\]
  \[I_{y'} \cup V(P_{i(y')}), J', \{v'_1\}, \ldots, \{v'_\ell\}, I_y, \{v_1\}, \ldots, \{v_\ell\}, \{y\}, \{s_{\eta+1}\}, \ldots, \{s_{\eta+\ell}\}\}\]
  with $R$ being the branch set of the root of (the subdivision of) $T'_{4,\ell}$; see~\cref{fig:T4l-zigzagged}.
  The fact that the branch sets are pairwise disjoint and non-adjacent when corresponding to non-adjacent vertices of~$T'_{4,\ell}$ is due to the $(\ell+1)$-ampleness of~$\cro$. 
  \begin{figure}[!ht]
  \centering
  \begin{tikzpicture}[vertex/.style={draw,circle,inner sep=0.06cm}]
    \def\s{1.2}
    \def\t{0.2}
    \def\l{5}
    \pgfmathsetmacro{\h}{1/(\l+1))}
    \pgfmathsetmacro{\hp}{1/\l)}
    \pgfmathtruncatemacro{\lm}{\l-1}
    \pgfmathtruncatemacro{\lmm}{\lm-1}
    
    \foreach \i/\j/\l/\lab in {0/0/r/R, -3/-1.5/z1/{V(P_1)},-2/-1.5/z2/,-1/-1.5/z3/,0/-1.5/z4/,1/-1.5/z5/{V(P_{\ell+2})}, 2.2/-1.5/z6/,3.2/-1.5/z7/}{
      \node[vertex] (\l) at (\i * \s,\j * \s) {} ;
      \node[blue] at ($(\l) + (0,0.35)$) {$\lab$};
    }

    \node[blue] at ($(z6) + (0.1,0.35)$) {$V(P_{i(x')})$};
    \node[blue] at ($(z7) + (1,0.35)$) {$I_{y'} \cup V(P_{i(y')})$};

    \foreach \i in {1,...,7}{
      \draw (r) -- (z\i) ;
    }

    \foreach \i in {1,...,5}{
      \draw (z\i) --++ (2 * \t,-1 * \t) --++ (0.5 * \t,-2.5 * \t) --++ (-5 * \t,0) --++ (0.5 * \t, 2.5 * \t) -- (z\i) ;
      \node at ([yshift=-0.45cm]z\i) {$T_{1,\ell}$};
    }

    \node[vertex] (x) at (5 * \s,-4 * \s) {} ;
    \node[blue] at ($(x) + (0,0.35)$) {$\{x\}$};
    \foreach \p in {1,...,\l}{
        \path (z6) to node[vertex, pos=\p * \h] (a\p) {} (x) ;
    }
    \node[blue] at ($(a\l) + (0,0.35)$) {$I_x$};
    \node[blue] at ($(a1) + (0,0.35)$) {$I_{x'}$};
    
    \node[vertex] (y) at (7 * \s,-4 * \s) {} ;
    \node[blue] at ($(y) + (0,0.35)$) {$\{y\}$};
    \foreach \p in {1,...,\l}{
        \path (z7) to node[vertex, pos=\p * \h] (b\p) {} (y) ;
    }
    \node[blue] at ($(b\l) + (0,0.35)$) {$I_y$};
    
    \draw (z6) -- (a1) ;
    \draw (z7) -- (b1) ;
    \foreach \p [count = \pp from 2] in {1,...,\lm}{
      \draw (a\p) -- (a\pp) ;
      \draw (b\p) -- (b\pp) ;
    }
    \draw (a\l) -- (x) ;
    \draw (b\l) -- (y) ;

    \foreach \i in {x,y}{
      \draw (\i) --++ (2 * \t,-1 * \t) --++ (0.5 * \t,-2.5 * \t) --++ (-5 * \t,0) --++ (0.5 * \t, 2.5 * \t) -- (\i) ;
      \node at ([yshift=-0.45cm]\i) {$T_{1,\ell}$};
    }

    \begin{scope}[rotate=-40]
    \foreach \i in {z6}{
      \draw (\i) --++ (2 * \t,-1 * \t) --++ (0.5 * \t,-2.5 * \t) --++ (-5 * \t,0) --++ (0.5 * \t, 2.5 * \t) -- (\i) ;
      \node at ([yshift=-0.45cm]\i) {$T_{1,\ell}$};
    }
    \end{scope}

    \node[vertex] (xp) at (3.8 * \s, -4.5 * \s) {} ;
    \foreach \p in {1,...,\lm}{
        \path (a5) to node[vertex, pos=\p * \hp] (c\p) {} (xp) ;
    }

    \node[vertex] (yp) at (5.8 * \s, -4.5 * \s) {} ;
    \foreach \p in {1,...,\lm}{
        \path (b5) to node[vertex, pos=\p * \hp] (d\p) {} (yp) ;
    }

    \node[vertex] (zp) at (5 * \s, -2 * \s) {} ;
    \foreach \p in {1,...,\lm}{
        \path (z7) to node[vertex, pos=\p * \hp] (e\p) {} (zp) ;
    }
    
    \node[vertex] (zzp) at (2 * \s, -2.9 * \s) {} ;
    \foreach \p in {1,...,\lm}{
        \path (a1) to node[vertex, pos=\p * \hp] (f\p) {} (zzp) ;
    }

    \draw (a5) -- (c1) ;
    \draw (b5) -- (d1) ;
    \draw (z7) -- (e1) ;
    \draw (f1) -- (a1) ;
    \foreach \p [count = \pp from 2] in {1,...,\lmm}{
      \draw (c\p) -- (c\pp) ;
      \draw (d\p) -- (d\pp) ;
      \draw (e\p) -- (e\pp) ;
      \draw (f\p) -- (f\pp) ;
    }
    \draw (c4) -- (xp) ;
    \draw (d4) -- (yp) ;
    \draw (e4) -- (zp) ;
    \draw (f4) -- (zzp) ;

    \begin{scope}[blue]
    \node at (1,-3.1) {$\{w_1\}, \ldots, \{w_\ell\}$} ;
    \node at (3.1,-4) {$\{u'_1\}, \ldots, \{u'_\ell\}, J$} ;    
    \node at (3.2,-5) {$\{u_1\}, \ldots, \{u_\ell\}$} ;
    \node at (7.5,-2.9) {$\{v'_1\}, \ldots, \{v'_\ell\}, J'$} ;
    \node at (7,-2) {$\{w'_1\}, \ldots, \{w'_\ell\}$} ;
    \node at (7,-5.8) {$\{v_1\}, \ldots, \{v_\ell\}$} ;
    \end{scope}
  \end{tikzpicture}
  \caption{The mapping of the branch sets of~$\mathcal M$ in $T_{4,\ell}$.}
  \label{fig:T4l-zigzagged}
\end{figure}
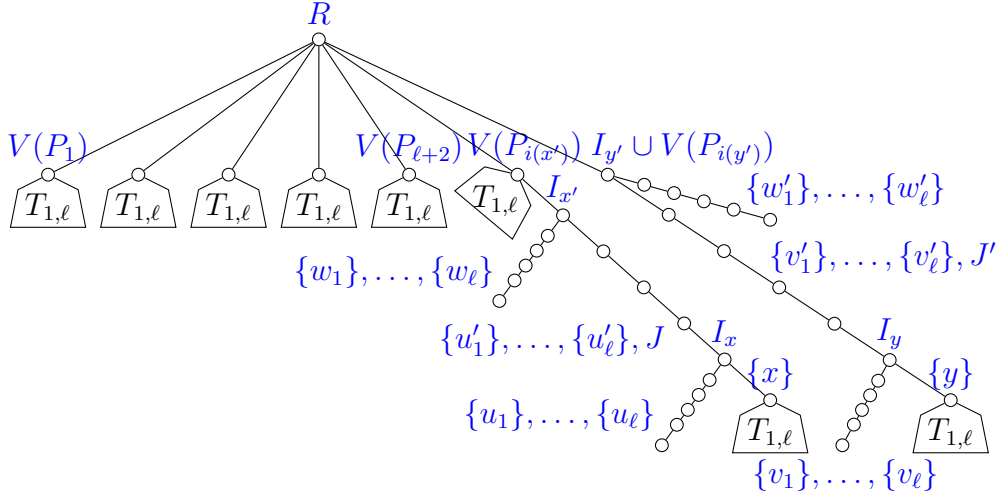
  
  From $\mathcal M$ we get an induced minor model of~$T_{4,\ell} \cup \ell T_{1,\ell}$ in $\cro$ by iteratively applying \cref{lem:induced-minor-extension} $(2\ell+3)\ell$ times; $\ell$ times to each of \[\{x\}, \{y\}, I_{x'} \cup V(P_{i(x')}), V(P_i)~\text{for}~i \in [\ell+2] \setminus \{i(x'),i(y')\}, \{s_{\eta+1}\}, \ldots, \{s_{\eta+\ell}\}.\]
  Note that $|\mathcal L_\cro|=(2\ell+3)\ell+2$ allows these $(\ell+3)\ell+\ell^2$ applications (with $Q$ and $Q'$ taking two paths of $\mathcal L_\cro$).
\end{proof}

\begin{lemma}\label{lem:t2l-in-interrupted}
  For any integers $\ell \geqslant 2$, any interrupted $(\ell+1)$-ample $(2\ell+4,(2\ell+3)\ell+1)$-constellation contains $T_{2,\ell} \cup \ell T_{1,\ell}$ and $T_{3,\ell} \cup \ell T_{1,\ell}$ as induced minors. 
\end{lemma}

\begin{proof}
  Let $\cro = (S_\cro, \mathcal L_\cro)$ be an interrupted $(\ell+1)$-ample $(2\ell+4,(2\ell+3)\ell+1)$-constellation with total order $\prec$ on $S_\cro$.
  Let $s_1 \prec \ldots \prec s_{2 \ell +4}$ enumerate $S_\cro$.
  We set $x := s_{\ell+1}$, $y := s_{\ell+2}$, $z := s_{2\ell+4}$, and $S_m := \{s_1, \ldots, s_{\ell}\}$.
  Let $P$ be a~shortest $\cro$-route between $x$ and $y$ using path $Q \in \mathcal L_\cro$.
  As $\cro$ is interrupted and $P$ is shortest, $V(P)$ and $S_m$ are non-adjacent.
  Let $I_x$ (resp.~$I_y$) be the interval of the neighbor of $x$ (resp.~of~$y$) on~$P$.
  Let $u_1, \ldots, u_\ell$ (resp.~$u'_1, \ldots, u'_\ell$) be the vertices of $Q$ just before (resp.~just after) $I_x$.
  Let $v_1, \ldots, v_\ell$ (resp.~$v'_1, \ldots, v'_\ell$) be the vertices of $Q$ just after (resp.~just before) $I_y$.
  
  As $\cro$ is interrupted, every vertex of $S_\cro \setminus S_m$ has a~neighbor in~$V(P)$.
  Let $x' \in V(P)$ (resp.~$y' \in V(P)$) be the first (resp.~last) vertex of~$P$ with a~neighbor in $S_\cro \setminus \{x,y\}$.
  Again because $\cro$ is interrupted, the unique neighbor of~$x'$ (resp.~$y'$) in $S_\cro$ has to be~$z$.
  Let $I_{x'}$ (resp.~$I_{y'}$) be the interval of~$x'$ (resp.~of~$y'$).

  We first build an induced minor model of~$T_{2,\ell} \cup \ell T_{1,\ell}$.
  Let $w_1, \ldots, w_\ell$ be the vertices of~$P$ just after $I_{x'}$.
  Let $R$ be the vertex set of the subpath of~$P$ between the first vertex after $w_\ell$ with a~neighbor in~$S_\cro$ (note that by $(\ell+1)$-ampleness this vertex is not adjacent to~$w_\ell$) and the last vertex of~$I_{y'}$.
  We observe that $N[R] \cap S_\cro = S_\cro \setminus (S_m \cup \{x,y\})$.
  Let $x_1, \ldots, x_\ell$ be $\ell$~distinct vertices of~$S_\cro \setminus (S_m \cup \{x,y,z\})$.
  We denote by $J$ (resp.~$J'$) the subset of vertices of~$P$ strictly between $\{u'_1, \ldots, u'_\ell\}$ and $I_{x'}$ (resp.~strictly between $R$ and $\{v'_1, \ldots, v'_\ell\}$).
  See~\cref{fig:constr-interrupted-T2l}.
  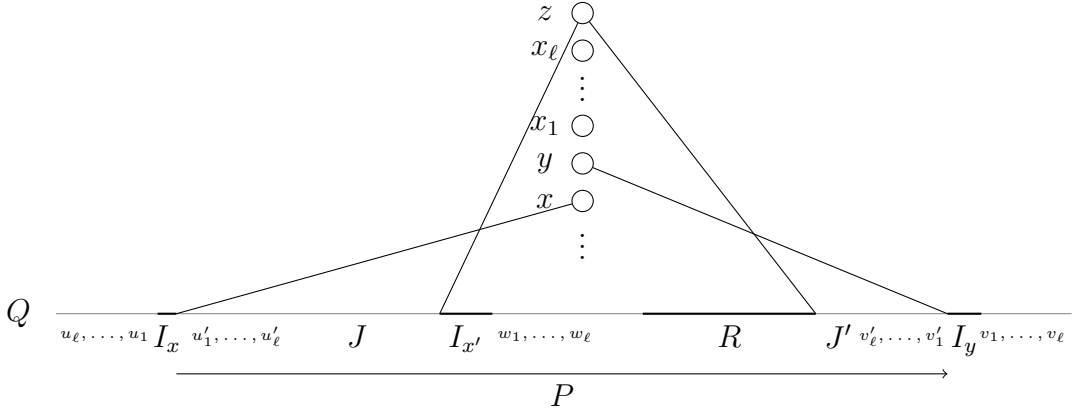
\begin{figure}[!ht]
  \centering
  \begin{tikzpicture}[vertex/.style={draw,circle,inner sep=0.1cm}]
    \foreach \i/\lab/\l in {-0.5/x/x,0/y/y,2/z/z,0.5/{x_1}/x1,1.5/{x_\ell}/xl}{
      \node[vertex] (\l) at (6,\i) {} ;
      \node at (5.5,\i) {$\lab$} ;
    }

    \node at (6,1.1) {$\vdots$} ;
    \node at (6,-1) {$\vdots$} ;
    
    \node at (-1.5,-2) {$Q$} ;
    \draw[gray,very thin] (-1,-2) -- (12.5,-2) ;

    \draw (x) -- (0.6,-2) ;
    \draw (z) -- (4.1,-2) ;
    \draw (z) -- (9.1,-2) ;
    \draw (y) -- (10.85,-2) ;

    \draw[->] (0.6,-2.8) to node[below,midway] {$P$} (10.85,-2.8) ;

    \draw[thick] (0.35,-2) --++(0.25,0) node[midway,below] {$I_x$};
    \draw[thick] (4.1,-2) --++(0.7,0) node[midway,below] {$I_{x'}$};
    \draw[thick] (10.85,-2) --++(0.45,0) node[midway,below] {$I_y$};
    \draw[thick] (6.8,-2) -- (9.1,-2) node[midway,below] {$R$};

    \node at (-0.35,-2.3) {\tiny{$u_\ell, \ldots, u_1$}} ;
    \node at (1.4,-2.3) {\tiny{$u'_1, \ldots, u'_\ell$}} ;
    \node at (3,-2.3) {$J$} ;

    \node at (5.5,-2.3) {\tiny{$w_1, \ldots, w_\ell$}} ;
    \node at (10.25,-2.3) {\tiny{$v'_\ell, \ldots, v'_1$}} ;
    \node at (11.85,-2.3) {\tiny{$v_1, \ldots, v_\ell$}} ;
    \node at (9.4,-2.3) {$J'$} ;
    
  \end{tikzpicture}
  \caption{The vertices and subsets used in the subsequent induced minor model of~$T'_{2,\ell}$.}
  \label{fig:constr-interrupted-T2l}
  \end{figure}
  
  We get an induced minor model $\mathcal M$ of an induced subdivision of~$T'_{2,\ell} \cup \ell K_1$ as follows:
  \[\mathcal M := \{\{x\}, \{u_1\}, \ldots, \{u_\ell\}, I_x, \{u'_1\}, \ldots, \{u'_\ell\}, J, \{w_1\}, \ldots, \{w_\ell\}, I_{x'}, \{z\},  R, \{x_1\}, \ldots, \{x_\ell\},\]
  \[J', \{v'_1\}, \ldots, \{v'_\ell\}, I_y, \{v_1\}, \ldots, \{v_\ell\}, \{y\}, \{s_1\}, \ldots, \{s_\ell\}\}\]
  with $R$ being the branch set of the root; see~\cref{fig:T2l-interrupted}.
  We then get an induced minor model of $T_{2,\ell} \cup \ell T_{1,\ell}$ as before, by applying~\cref{lem:induced-minor-extension} $(2\ell+3)\ell$ times; $\ell$ times to each branch set of $\{x\}$, $\{y\}$, $I_{x'} \cup \{z\}$, $\{x_1\}, \ldots, \{x_\ell\}, \{s_1\}, \ldots, \{s_\ell\}$.
  \begin{figure}[!ht]
  \centering
  \begin{tikzpicture}[vertex/.style={draw,circle,inner sep=0.06cm},scale=.9]
    \def\s{1.2}
    \def\t{0.2}
    \def\l{5}
    \pgfmathsetmacro{\h}{1/(\l+1))}
    \pgfmathsetmacro{\hp}{1/\l)}
    \pgfmathtruncatemacro{\lm}{\l-1}
    \pgfmathtruncatemacro{\lmm}{\lm-1}
    
    \foreach \i/\j/\l/\lab in {0/0/r/R, -3/-1.5/z1/{\{x_1\}},-2/-1.5/z2/{\{x_2\}},-1/-1.5/z3/{\{x_3\}},0/-1.5/z4/{\ldots},1/-1.5/z5/{\{x_\ell\}}, 2.2/-1.5/z6/{\{z\}},3.2/-1.5/z7/}{
      \node[vertex] (\l) at (\i * \s,\j * \s) {} ;
      \node[blue] at ($(\l) + (0,0.35)$) {$\lab$};
    }

    \foreach \i in {1,...,7}{
      \draw (r) -- (z\i) ;
    }

    \foreach \i in {1,...,5}{
      \draw (z\i) --++ (2 * \t,-1 * \t) --++ (0.5 * \t,-2.5 * \t) --++ (-5 * \t,0) --++ (0.5 * \t, 2.5 * \t) -- (z\i) ;
      \node at ([yshift=-0.45cm]z\i) {$T_{1,\ell}$};
    }

    \node[vertex] (x) at (5 * \s,-4 * \s) {} ;
    \node[blue] at ($(x) + (0,0.35)$) {$\{x\}$};
    \foreach \p in {1,...,\l}{
        \path (z6) to node[vertex, pos=\p * \h] (a\p) {} (x) ;
    }
    \node[blue] at ($(a\l) + (0,0.35)$) {$I_x$};
    \node[blue] at ($(a1) + (0,0.35)$) {$I_{x'}$};
    
    \node[vertex] (y) at (7 * \s,-4 * \s) {} ;
    \node[blue] at ($(y) + (0,0.35)$) {$\{y\}$};
    \foreach \p in {1,...,\l}{
        \path (z7) to node[vertex, pos=\p * \h] (b\p) {} (y) ;
    }
    \node[blue] at ($(b\l) + (0,0.35)$) {$I_y$};
    
    \draw (z6) -- (a1) ;
    \draw (z7) -- (b1) ;
    \foreach \p [count = \pp from 2] in {1,...,\lm}{
      \draw (a\p) -- (a\pp) ;
      \draw (b\p) -- (b\pp) ;
    }
    \draw (a\l) -- (x) ;
    \draw (b\l) -- (y) ;

    \foreach \i in {x,y}{
      \draw (\i) --++ (2 * \t,-1 * \t) --++ (0.5 * \t,-2.5 * \t) --++ (-5 * \t,0) --++ (0.5 * \t, 2.5 * \t) -- (\i) ;
      \node at ([yshift=-0.45cm]\i) {$T_{1,\ell}$};
    }

    \begin{scope}[rotate=-40]
    \foreach \i in {z6}{
      \draw (\i) --++ (2 * \t,-1 * \t) --++ (0.5 * \t,-2.5 * \t) --++ (-5 * \t,0) --++ (0.5 * \t, 2.5 * \t) -- (\i) ;
      \node at ([yshift=-0.45cm]\i) {$T_{1,\ell}$};
    }
    \end{scope}

    \node[vertex] (xp) at (3.8 * \s, -4.5 * \s) {} ;
    \foreach \p in {1,...,\lm}{
        \path (a5) to node[vertex, pos=\p * \hp] (c\p) {} (xp) ;
    }

    \node[vertex] (yp) at (5.8 * \s, -4.5 * \s) {} ;
    \foreach \p in {1,...,\lm}{
        \path (b5) to node[vertex, pos=\p * \hp] (d\p) {} (yp) ;
    }

    \node[vertex] (zp) at (2 * \s, -2.9 * \s) {} ;
    \foreach \p in {1,...,\lm}{
        \path (a1) to node[vertex, pos=\p * \hp] (e\p) {} (zp) ;
    }

    \draw (a5) -- (c1) ;
    \draw (b5) -- (d1) ;
    \draw (a1) -- (e1) ;
    \foreach \p [count = \pp from 2] in {1,...,\lmm}{
      \draw (c\p) -- (c\pp) ;
      \draw (d\p) -- (d\pp) ;
      \draw (e\p) -- (e\pp) ;
    }
    \draw (c4) -- (xp) ;
    \draw (d4) -- (yp) ;
    \draw (e4) -- (zp) ;

    \begin{scope}[blue]
    \node at (1,-3.1) {$\{w_1\}, \ldots, \{w_\ell\}$} ;
    \node at (3.2,-5) {$\{u_1\}, \ldots, \{u_\ell\}$} ;
    \node at (7.2,-2.6) {$\{v'_1\}, \ldots, \{v'_\ell\}, J'$} ;
    \node at (7,-5.8) {$\{v_1\}, \ldots, \{v_\ell\}$} ;
    \node at (3,-4) {$\{u'_1\}, \ldots, \{u'_\ell\}, J$} ;
    \end{scope}
  \end{tikzpicture}
  \caption{The mapping of the branch sets of~$\mathcal M$ in $T_{2,\ell}$ (without the $\ell$ isolated vertices).}
  \label{fig:T2l-interrupted}
\end{figure}
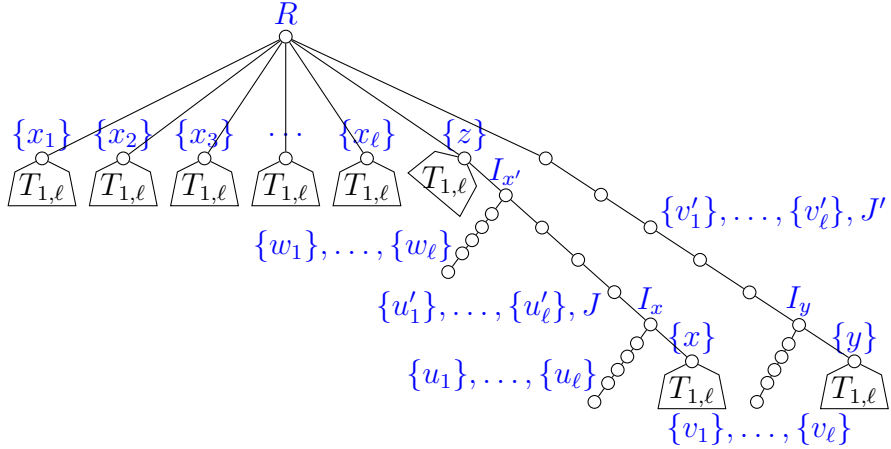

  We now build an induced minor model of~$T_{3,\ell} \cup \ell T_{1,\ell}$.
  Let $L$ be the set of vertices of~$P$ strictly between $u'_\ell$ and the first vertex after $u'_\ell$ with a~neighbor $z'$ in $S_\cro \setminus \{z\}$.
  We rename $x_1, \ldots, x_\ell$ the vertices of $S_\cro \setminus \{x,y,z,z'\}$.
  Let $R'$ be the vertex set of the minimal subpath of~$P$ starting just after $L$ and being adjacent to all the vertices of $S_\cro \setminus \{x,y,z,z'\}$.
  Let $w'_1, \ldots, w'_\ell$ be the vertices of~$P$ just before $I_{y'}$.
  Let $J''$ be the set of vertices of~$P$ strictly between $I_{y'}$ and $v'_1, \ldots, v'_\ell$.
  See~\cref{fig:constr-interrupted-T3l}.
   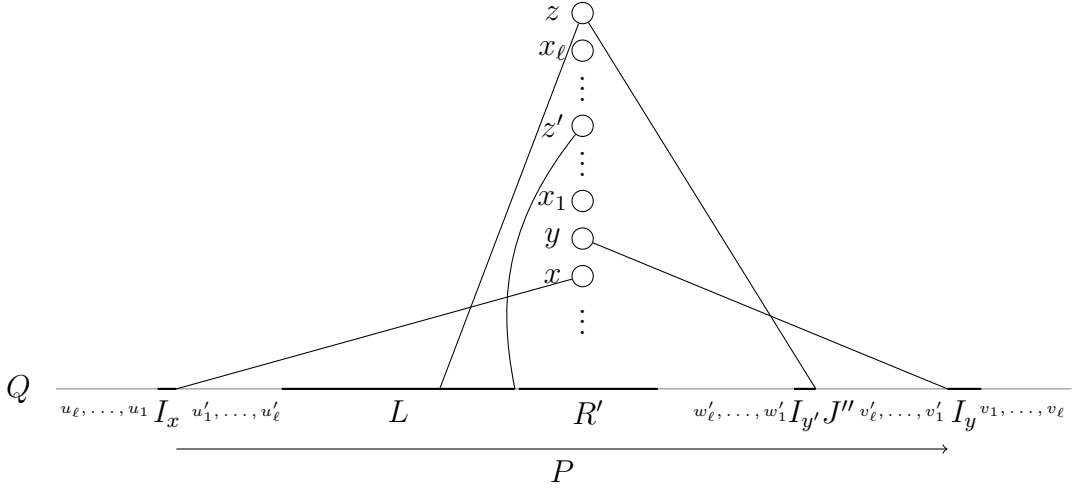
\begin{figure}[!ht]
  \centering
  \begin{tikzpicture}[vertex/.style={draw,circle,inner sep=0.1cm}]
    \foreach \i/\lab/\l in {-0.5/x/x,0/y/y,3/z/z,0.5/{x_1}/x1,2.5/{x_\ell}/xl,1.5/{z'}/zp}{
      \node[vertex] (\l) at (6,\i) {} ;
      \node at (5.6,\i) {$\lab$} ;
    }

    \node at (6,-1) {$\vdots$} ;
    \node at (6,2.1) {$\vdots$} ;
    \node at (6,1.1) {$\vdots$} ;
    \node at (-1.5,-2) {$Q$} ;
    \draw[gray,very thin] (-1,-2) -- (12.5,-2) ;

    \draw (x) -- (0.6,-2) ;
    \draw (z) -- (4.1,-2) ;
    \draw (z) -- (9.1,-2) ;
    \draw (y) -- (10.85,-2) ;
    \draw (zp) to [bend left=-25] (5.1,-2) ;

    \draw[->] (0.6,-2.8) to node[below,midway] {$P$} (10.85,-2.8) ;

    \draw[thick] (0.35,-2) --++(0.25,0) node[midway,below] {$I_x$};
    \draw[thick] (2,-2) -- (5.1,-2) node[midway,below] {$L$};
    \draw[thick] (5.15,-2) -- (7,-2) node[midway,below] {$R'$};
    \draw[thick] (8.81,-2) --++(0.29,0) node[midway,below] {$I_{y'}$};
    \draw[thick] (10.85,-2) --++(0.45,0) node[midway,below] {$I_y$};

    \node at (-0.35,-2.3) {\tiny{$u_\ell, \ldots, u_1$}} ;
    \node at (1.4,-2.3) {\tiny{$u'_1, \ldots, u'_\ell$}} ;

    \node at (8.1,-2.3) {\tiny{$w'_\ell, \ldots, w'_1$}} ;
    \node at (10.25,-2.3) {\tiny{$v'_\ell, \ldots, v'_1$}} ;
    \node at (11.85,-2.3) {\tiny{$v_1, \ldots, v_\ell$}} ;
    \node at (9.4,-2.3) {$J''$} ;
    
  \end{tikzpicture}
  \caption{The vertices and subsets used in the subsequent induced minor model of~$T'_{3,\ell}$.}
  \label{fig:constr-interrupted-T3l}
  \end{figure}

  We get an induced minor model $\mathcal M'$ of an induced subdivision of~$T'_{3,\ell} \cup \ell K_1$ as follows:
  \[\mathcal M' := \{\{x\}, \{u_1\}, \ldots, \{u_\ell\}, I_x, \{u'_1\}, \ldots, \{u'_\ell\}, L \cup \{z'\}, R' \cup \{z\}, \{x_1\}, \ldots, \{x_\ell\},\]
   \[\{w'_1\}, \ldots, \{w'_\ell\}, I_{y'}, J'', \{v'_1\}, \ldots, \{v'_\ell\}, I_y, \{v_1\}, \ldots, \{v_\ell\}, \{y\}, \{s_1\}, \ldots, \{s_\ell\}\}\]
   with $R' \cup \{z\}$ the branch set of the root; see~\cref{fig:T3l-interrupted}.
     \begin{figure}[!ht]
  \centering
  \begin{tikzpicture}[vertex/.style={draw,circle,inner sep=0.06cm},scale=.9]
    \def\s{1.2}
    \def\t{0.2}
    \def\l{5}
    \pgfmathsetmacro{\h}{1/(\l+1))}
    \pgfmathsetmacro{\hp}{1/\l)}
    \pgfmathtruncatemacro{\lm}{\l-1}
    \pgfmathtruncatemacro{\lmm}{\lm-1}
    
    \foreach \i/\j/\l/\lab in {0/0/r/{R' \cup \{z\}}, -3/-1.5/z1/{\{x_1\}},-2/-1.5/z2/{\{x_2\}},-1/-1.5/z3/{\{x_3\}},0/-1.5/z4/{\ldots},1/-1.5/z5/{\{x_\ell\}}, 2.2/-1.5/z6/{L \cup \{z'\}},3.2/-1.5/z7/{I_{y'}}}{
      \node[vertex] (\l) at (\i * \s,\j * \s) {} ;
      \node[blue] at ($(\l) + (0,0.35)$) {$\lab$};
    }

    \foreach \i in {1,...,7}{
      \draw (r) -- (z\i) ;
    }

    \foreach \i in {1,...,5}{
      \draw (z\i) --++ (2 * \t,-1 * \t) --++ (0.5 * \t,-2.5 * \t) --++ (-5 * \t,0) --++ (0.5 * \t, 2.5 * \t) -- (z\i) ;
      \node at ([yshift=-0.45cm]z\i) {$T_{1,\ell}$};
    }

    \node[vertex] (x) at (5 * \s,-4 * \s) {} ;
    \node[blue] at ($(x) + (0,0.35)$) {$\{x\}$};
    \foreach \p in {1,...,\l}{
        \path (z6) to node[vertex, pos=\p * \h] (a\p) {} (x) ;
    }
    \node[blue] at ($(a\l) + (0,0.35)$) {$I_x$};
    
    \node[vertex] (y) at (7 * \s,-4 * \s) {} ;
    \node[blue] at ($(y) + (0,0.35)$) {$\{y\}$};
    \foreach \p in {1,...,\l}{
        \path (z7) to node[vertex, pos=\p * \h] (b\p) {} (y) ;
    }
    \node[blue] at ($(b\l) + (0,0.35)$) {$I_y$};
    
    \draw (z6) -- (a1) ;
    \draw (z7) -- (b1) ;
    \foreach \p [count = \pp from 2] in {1,...,\lm}{
      \draw (a\p) -- (a\pp) ;
      \draw (b\p) -- (b\pp) ;
    }
    \draw (a\l) -- (x) ;
    \draw (b\l) -- (y) ;

    \foreach \i in {x,y}{
      \draw (\i) --++ (2 * \t,-1 * \t) --++ (0.5 * \t,-2.5 * \t) --++ (-5 * \t,0) --++ (0.5 * \t, 2.5 * \t) -- (\i) ;
      \node at ([yshift=-0.45cm]\i) {$T_{1,\ell}$};
    }

    \begin{scope}[rotate=-40]
    \foreach \i in {z6}{
      \draw (\i) --++ (2 * \t,-1 * \t) --++ (0.5 * \t,-2.5 * \t) --++ (-5 * \t,0) --++ (0.5 * \t, 2.5 * \t) -- (\i) ;
      \node at ([yshift=-0.45cm]\i) {$T_{1,\ell}$};
    }
    \end{scope}

    \node[vertex] (xp) at (3.8 * \s, -4.5 * \s) {} ;
    \foreach \p in {1,...,\lm}{
        \path (a5) to node[vertex, pos=\p * \hp] (c\p) {} (xp) ;
    }

    \node[vertex] (yp) at (5.8 * \s, -4.5 * \s) {} ;
    \foreach \p in {1,...,\lm}{
        \path (b5) to node[vertex, pos=\p * \hp] (d\p) {} (yp) ;
    }

    \node[vertex] (zp) at (5 * \s, -2 * \s) {} ;
    \foreach \p in {1,...,\lm}{
        \path (z7) to node[vertex, pos=\p * \hp] (e\p) {} (zp) ;
    }

    \draw (a5) -- (c1) ;
    \draw (b5) -- (d1) ;
    \draw (z7) -- (e1) ;
    \foreach \p [count = \pp from 2] in {1,...,\lmm}{
      \draw (c\p) -- (c\pp) ;
      \draw (d\p) -- (d\pp) ;
      \draw (e\p) -- (e\pp) ;
    }
    \draw (c4) -- (xp) ;
    \draw (d4) -- (yp) ;
    \draw (e4) -- (zp) ;

    \begin{scope}[blue]
    \node at (3.2,-5) {$\{u_1\}, \ldots, \{u_\ell\}$} ;
    \node at (7.5,-2.9) {$\{v'_1\}, \ldots, \{v'_\ell\}, J''$} ;
    \node at (7,-2) {$\{w'_1\}, \ldots, \{w'_\ell\}$} ;
    \node at (7,-5.8) {$\{v_1\}, \ldots, \{v_\ell\}$} ;
    \node at (2.45,-3.5) {$\{u'_1\}, \ldots, \{u'_\ell\}$} ;
    \end{scope}
  \end{tikzpicture}
  \caption{The mapping of the branch sets of~$\mathcal M'$ in $T_{3,\ell}$.}
  \label{fig:T3l-interrupted}
     \end{figure}
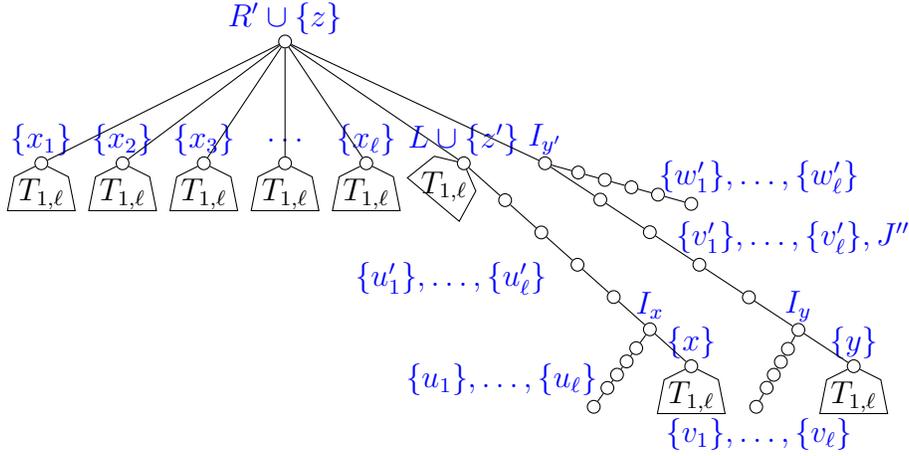
     We finally get an induced minor model of $T_{3,\ell} \cup \ell T_{1,\ell}$ by applying~\cref{lem:induced-minor-extension} $(2\ell+3)\ell$ times; $\ell$ times to each branch set of $\{x\}$, $\{y\}$, $L \cup \{z'\}$, $\{x_1\}, \ldots, \{x_\ell\}, \{s_1\}, \ldots, \{s_\ell\}$.
\end{proof}

\cref{lem:pw-constell,lem:t2l-in-zigzagged,lem:t2l-in-interrupted} implies the following.
\begin{lemma}\label{lem:posi}
  For any integers $\ell > t \geqslant 1$, there is an integer $w := w(\ell,t)$ such that every graph of pathwidth at~least~$w$ with no $K_{t,t}$ subgraph admits $T_{2,\ell} \cup \ell T_{1,\ell}$ and $T_{3,\ell} \cup \ell T_{1,\ell}$ as induced minors.
  Hence, every forest that is an induced minor of some $T_{2,\ell} \cup \ell T_{1,\ell}$ or $T_{3,\ell} \cup \ell T_{1,\ell}$ is \posi.
\end{lemma}

Finally, \cref{lem:posi,lem:nega2} establish~\cref{thm:main-alt}.
Recall indeed that any graph that is not a~forest is \nega.
Besides, in~\cref{sec:negative-forests} we showed that every forest that is not an induced minor of any $T_{2,\ell} \cup \ell T_{1,\ell}$ or $T_{3,\ell} \cup \ell T_{1,\ell}$ satisfies the preconditions of~\cref{lem:npd}, which corresponds for forests to the first item of~\cref{thm:main}, or the preconditions of~\cref{lem:nds}, which corresponds to the second item of~\cref{thm:main}.
As we have proved in this section that all the other forests are \posi, they cannot satisfy either item, which finishes the proof of~\cref{thm:main}.

\bibliographystyle{abbrv}

\begin{thebibliography}{10}

\bibitem{istd-series}
T.~Abrishami, B.~Alecu, M.~Chudnovsky, C.~Dibek, P.~Gartland, S.~Hajebi,
  D.~Lokshtanov, P.~Rz{\k{a}}{\.z}ewski, S.~Spirkl, and K.~Vušković.
\newblock Induced subgraphs and tree decompositions {I}--{XIX}, 2020-.

\bibitem{AbrishamiACHS24}
T.~Abrishami, B.~Alecu, M.~Chudnovsky, S.~Hajebi, and S.~Spirkl.
\newblock Induced subgraphs and tree decompositions {VII.} {B}asic obstructions
  in \emph{H}-free graphs.
\newblock {\em Journal of Combinatorial Theory, Series B}, 164:443--472, 2024.

\bibitem{AGHK2025EP}
J.~Ahn, J.~P. Gollin, T.~Huynh, and O.~joung Kwon.
\newblock A coarse {E}rd{\H{o}}s--{P}\'osa theorem.
\newblock In {\em Proc. of the 2025 Annual ACM-SIAM Symp. on Discrete
  Algorithms (SODA)}, pages 3363--3381. SIAM, 2025.

\bibitem{Bonamy24}
M.~Bonamy, {\'{E}}.~Bonnet, H.~Déprés, L.~Esperet, C.~Geniet, C.~Hilaire,
  S.~Thomassé, and A.~Wesolek.
\newblock Sparse graphs with bounded induced cycle packing number have
  logarithmic treewidth.
\newblock {\em Journal of Combinatorial Theory, Series B}, 167:215--249, 2024.

\bibitem{Bonnet25}
{\'{E}}.~Bonnet.
\newblock Sparse induced subgraphs of large treewidth.
\newblock {\em Journal of Combinatorial Theory, Series B}, 173:184--203, 2025.

\bibitem{istd18}
M.~Chudnovsky, S.~Hajebi, and S.~Spirkl.
\newblock Induced subgraphs and tree decompositions {XVIII}. {O}bstructions to
  bounded pathwidth.
\newblock {\em arXiv preprint arXiv:2412.17756}, 2024.

\bibitem{istd16}
M.~Chudnovsky, S.~Hajebi, and S.~Spirkl.
\newblock Induced subgraphs and tree decompositions {XVI.} {C}omplete bipartite
  induced minors.
\newblock {\em Journal of Combinatorial Theory, Series B}, 176:287--318, 2026.

\bibitem{Courcelle90}
B.~Courcelle.
\newblock The monadic second-order logic of graphs. {I. R}ecognizable sets of
  finite graphs.
\newblock {\em Inf. Comput.}, 85(1):12--75, 1990.

\bibitem{Dallard25}
C.~Dallard, M.~Dumas, C.~Hilaire, and A.~Perez.
\newblock Sufficient conditions for polynomial-time detection of induced
  minors.
\newblock In R.~Kr{\'{a}}lovic and V.~Kurkov{\'{a}}, editors, {\em {SOFSEM}
  2025: Theory and Practice of Computer Science - 50th International Conference
  on Current Trends in Theory and Practice of Computer Science, {SOFSEM} 2025,
  Bratislava, Slovak Republic, January 20-23, 2025, Proceedings, Part {I}},
  volume 15538 of {\em Lecture Notes in Computer Science}, pages 195--208.
  Springer, 2025.

\bibitem{Davies22}
J.~Davies.
\newblock Oberwolfach report 1/2022.
\newblock 2022.

\bibitem{Hickingbotham22}
R.~Hickingbotham.
\newblock Induced subgraphs and path decompositions.
\newblock {\em Electron. J. Combin.}, 30(2):P2.37, 2023.

\bibitem{Korhonen21}
T.~Korhonen.
\newblock A single-exponential time 2-approximation algorithm for treewidth.
\newblock In {\em 62nd {IEEE} Annual Symposium on Foundations of Computer
  Science, {FOCS} 2021, Denver, CO, USA, February 7-10, 2022}, pages 184--192.
  {IEEE}, 2021.

\bibitem{Korhonen23}
T.~Korhonen.
\newblock Grid induced minor theorem for graphs of small degree.
\newblock {\em Journal of Combinatorial Theory, Series B}, 160:206--214, 2023.

\bibitem{KL2024InducedMinor}
T.~Korhonen and D.~Lokshtanov.
\newblock Induced-minor-free graphs: {S}eparator theorem, subexponential
  algorithms, and improved hardness of recognition.
\newblock In {\em Proc. of the 2024 Annual ACM-SIAM Symp. on Discrete
  Algorithms (SODA)}, pages 5249--5275, 2024.

\bibitem{Pohoata14}
C.~Pohoata.
\newblock Unavoidable induced subgraphs of large graphs.
\newblock Senior thesis, Department of Mathematics, Princeton University, 2014.

\bibitem{RS1}
N.~Robertson and P.~Seymour.
\newblock {Graph minors. I. Excluding a forest}.
\newblock {\em Journal of Combinatorial Theory, Series B}, 35(1):39--61, 1983.

\bibitem{ROBERTSON198692}
N.~Robertson and P.~Seymour.
\newblock Graph minors. {V}. {E}xcluding a planar graph.
\newblock {\em Journal of Combinatorial Theory, Series B}, 41(1):92--114, 1986.

\end{thebibliography}

\end{document}